\documentclass[10pt, reqno]{amsart}

\def\BibTeX{{\rm B\kern-.05em{\sc i\kern-.025em b}\kern-.08em
    T\kern-.1667em\lower.7ex\hbox{E}\kern-.125emX}}

\newtheorem{thm}{Theorem}[section]
\newtheorem{lem}[thm]{Lemma}
\newtheorem{prop}[thm]{Proposition}
\theoremstyle{definition}

\theoremstyle{remark}
\newtheorem{rem}{Remark}[section]

\numberwithin{equation}{section}

    \newcommand{\floor}[1]{\lfloor#1\rfloor}

    \newcommand{\EE}{\mathbb{E}}

    \renewcommand{\Pr}{\operatorname{P}}

    \newcommand{\dto}{\xrightarrow{d}}
    
    \newcommand{\vto}{\xrightarrow{v}}

    \newcommand{\toi}{\to\infty}

    \newcommand{\eind}{\stackrel{d}{=}}
    \newcommand{\rmd}{\mathrm{d}}

\newcommand{\be}{\begin{equation}}
    \newcommand{\ee}{\end{equation}}

\begin{document}

\title[Functional convergence for moving averages] % running head version
{Functional convergence for moving averages with heavy tails and random coefficients}

%%
% See the complete version of this file, testart.tex, for more
% complicated examples
%
\author{Danijel Krizmani\'{c}}

\address{Danijel Krizmani\'{c}\\ Department of Mathematics\\
        University of Rijeka\\
        Radmile Matej\v{c}i\'{c} 2, 51000 Rijeka\\
        Croatia}
%% Note the doubled @@:
\email{dkrizmanic@math.uniri.hr}

%\thanks{Research supported in part by NSF grant
%CCR-87-10433 and DARPA Contract N00019-89-J-1988.}

%\date{July 25, 1994}


%\maketitle

\begin{abstract}
We study functional convergence of sums of moving averages with random coefficients and heavy-tailed innovations. Under some standard moment conditions and the assumption that all partial sums of the series of coefficients are a.s.~bounded between zero and the sum of the series we obtain functional convergence of the corresponding partial sum stochastic process in the space $D[0,1]$ of c\`{a}dl\`{a}g functions with the Skorohod $M_{2}$ topology.
\end{abstract}

\maketitle

\section{Introduction}
\label{intro}

Let $(Z_{i})_{i \in \mathbb{Z}}$ be a sequence of i.i.d.~regularly varying random variables with index of regular variation $\alpha \in (0,2)$.
This means that
\begin{equation}\label{e:regvar}
 \Pr(|Z_{i}| > x) = x^{-\alpha} L(x), \qquad x>0,
\end{equation}
where $L$ is a slowly varying function at $\infty$. Regular variation implies $\mathrm{E}|Z_{i}|^{\beta} < \infty$ for every $\beta \in (0,\alpha)$.
We study
the moving average process with random coefficients, defined by
\begin{equation}\label{e:MArandom}
X_{i} = \sum_{j=0}^{\infty}C_{j}Z_{i-j}, \qquad i \in \mathbb{Z},
\end{equation}
where $(C_{i})_{i \geq 0 }$ is a sequence of random variables independent of $(Z_{i})$, such that the series in (\ref{e:MArandom}) is a.s.~convergent. One sufficient condition for that is
\begin{equation}\label{e:asconv}
\sum_{j=0}^{\infty}|C_{j}|^{\alpha - \epsilon} < \infty \quad \textrm{a.s. for some} \ \epsilon >0
\end{equation}
(see Hult and Samorodnitsky~\cite{HuSa08}).
We will use the following moment condition on the sequence $(C_{j})$:
\begin{equation}\label{e:momcond}
\sum_{j=0}^{\infty} \mathrm{E} |C_{j}|^{\delta} < \infty \qquad \textrm{for some}  \ \delta < \alpha,\,0 < \delta \leq 1.
\end{equation}
This condition also implies the a.s.~convergence of the series in (\ref{e:MArandom}), since
$$ \mathrm{E}|X_{i}|^{\delta} \leq \sum_{j=0}^{\infty} \mathrm{E}|C_{j}|^{\delta} \mathrm{E}|Z_{i-j}|^{\delta} = \mathrm{E}|Z_{1}|^{\delta} \sum_{j=0}^{\infty}\mathrm{E}|C_{j}|^{\delta} < \infty.$$
Beside condition (\ref{e:momcond}) we will require some other moment conditions, which will be specified in Section~\ref{S:InfiniteMA}. We also impose the following (usual) regularity conditions on $Z_{1}$:
  \begin{eqnarray}\label{e:oceknula}
    \mathrm{E} Z_{1}=0, & & \textrm{if} \ \ \alpha \in (1,2),  \\
    Z_{1} \ \textrm{is symmetric}, & & \textrm{if} \ \ \alpha=1.\label{e:sim}
  \end{eqnarray}
Let $(a_{n})$ be a sequence of positive real numbers such that
\be\label{eq:niz}
n \Pr (|Z_{1}|>a_{n}) \to 1,
\ee
as $n \to \infty$. Regular
variation of $Z_{i}$ can be expressed in terms of
vague convergence of measures on $\EE = \overline{\mathbb{R}} \setminus \{0\}$: for $a_n$ as in
\eqref{eq:niz} and as $n \to \infty$,
\begin{equation}
  \label{eq:onedimregvar}
  n \Pr( a_n^{-1} Z_i \in \cdot \, ) \vto \mu( \, \cdot \,),
\end{equation}
with the measure $\mu$ on $\EE$  given by
\begin{equation}
\label{eq:mu}
  \mu(\rmd x) = \bigl( p \, 1_{(0, \infty)}(x) + r \, 1_{(-\infty, 0)}(x) \bigr) \, \alpha |x|^{-\alpha-1} \, \rmd x,
\end{equation}
where
\be\label{eq:pq}
p =   \lim_{x \to \infty} \frac{\Pr(Z_i > x)}{\Pr(|Z_i| > x)} \qquad \textrm{and} \qquad
  r =   \lim_{x \to \infty} \frac{\Pr(Z_i \leq -x)}{\Pr(|Z_i| > x)}.
\ee
When the coefficients $C_{i}$ are deterministic, Basrak and Krizmani\'{c}~\cite{BaKr} obtained functional convergence of the partial sum process of $X_{i}$'s with respect to the Skorohod $M_{2}$ topology on $D[0,1]$. More precisely, they showed that under the condition on the coefficients $C_{i}$:
\be\label{eq:InfiniteMAcond}
0 \le \sum_{i=0}^{s}C_{i} \Bigg/ \sum_{i=0}^{\infty}C_{i} \le 1, \qquad \textrm{for every} \ s=0, 1, 2 \ldots,
\ee
%\be\label{e:BaKr}
%\frac{1}{a_{n}}\sum_{i=1}^{\floor{n\,\cdot}} (X_{i} - c_{n}) \dto \bigg( \sum_{j=0}^{\infty}C_{j} \bigg) %V(\,\cdot\,),
%\ee
 the following
\be\label{e:BaKr}
\frac{1}{a_{n}}\sum_{i=1}^{\floor{n\,\cdot}} X_{i}  \dto \bigg( \sum_{j=0}^{\infty}C_{j} \bigg) V(\,\cdot\,),
\ee
holds in $D[0,1]$, where $V(\,\cdot\,)$ is an $\alpha$--stable L\'{e}vy process and $D[0,1]$
is the space of real--valued right continuous functions on $[0,1]$ with left limits.

Recall here that if at least two coefficients are nonzero, then the convergence in (\ref{e:BaKr}) cannot hold with respect to the more usual Skorohod $J_{1}$ topology on $D[0,1]$, but
 if all the coefficients are nonnegative, then the convergence in (\ref{e:BaKr}) holds in the $M_{1}$ topology, see
 Avram and Taqqu~\cite{AvTa92}.
 The aim of this article is to obtain the functional convergence with respect to the $M_{2}$ topology as in (\ref{e:BaKr}) when the coefficients $C_{i}$ are random variables. Limit theory for moving averages with random coefficients, but without the time component, have already been studied, see Kulik~\cite{Ku06}. These processes can represent various stochastic models, such are solutions to stochastic recurrence equations and stochastic integrals (usually with some predictability assumption instead of the independence between the coefficients $C_{j}$ and the noise variables $Z_{j}$, see Hult and Samorodnitsky~\cite{HuSa08}).
 
 The Skorohod $M_{2}$ topology on $D[0, 1]$ is defined using completed graphs and their parametric representations (see Section 12.11 in Whitt~\cite{Whitt02} for details). Here we give only a characterization of the $M_{2}$ topology using the Hausdorff metric on the spaces of graphs, since it will be convenient for our purposes. For $x_{1},x_{2} \in D[0,1]$ define
$$ d_{M_{2}}(x_{1}, x_{2}) = \bigg(\sup_{a \in \Gamma_{x_{1}}} \inf_{b \in \Gamma_{x_{2}}} d(a,b) \bigg) \vee \bigg(\sup_{a \in \Gamma_{x_{2}}} \inf_{b \in \Gamma_{x_{1}}} d(a,b) \bigg),$$
where $d$ is the metric on $\mathbb{R}^{2}$ defined by $d((x_{1},y_{1}),(x_{2},y_{2}))=|x_{1}-x_{2}| \vee |y_{1}-y_{2}|$ for $(x_{i},y_{i}) \in \mathbb{R}^{2},\,i=1,2$, where $a \vee b = \max\{a,b\}$.
The metric $d_{M_{2}}$ induces the $M_{2}$ topology. This topology is weaker than the more frequently used $M_{1}$ and $J_{1}$ topologies.

 The paper is organized as follows.
  In Section~\ref{S:FiniteMA} we obtain functional convergence for finite order moving average processes, and then in Section~\ref{S:InfiniteMA} we
extend this result to infinite order moving averages. A technical result needed for establishing functional convergence for infinite order moving averages when $\alpha \in [1,2)$ is given in Appendix.

\section{Finite order MA processes}
\label{S:FiniteMA}

 Let $C_{0}, C_{1}, \ldots , C_{q}$ (for some fixed $q \in \mathbb{N}$) be random variables satisfying
\be\label{eq:FiniteMAcond}
0 \le \sum_{i=0}^{s}C_{i} \Bigg/ \sum_{i=0}^{q}C_{i} \le 1 \ \ \textrm{a.s.} \qquad \textrm{for every} \ s=0, 1, \ldots, q.
\ee
Put $ C = \sum_{i=0}^{q}C_{i}$.
%Without loss of generality assume $\Phi > 0$.
%The case $\Phi<0$ is completely equivalent if we multiply the
%noise sequence $(Z_i)$ by minus 1, and is therefore omitted.
Observe that condition (\ref{eq:FiniteMAcond}) implies that $C$,
$ \sum_{i=0}^{s}C_{i}$ and $ \sum_{i=s}^{q}C_{i}$ are a.s.~of the same sign for every $ s=0,1,\ldots,q$. Also note that condition (\ref{eq:FiniteMAcond}) is satisfied if the $C_{j}$' are all nonnegative or all nonpositive.

%Let $(X_{t})$ be a moving average process defined by
%$$ X_{t} = \sum_{i=0}^{q}C_{i}Z_{t-i}, \qquad t \in \mathbb{Z},$$
%and let the corresponding partial sum process be
%\be\label{eq:defVn}
%V_{n}(t) = \frac{1}{a_{n}} \Bigg( \sum_{i=1}^{\floor {nt}}X_{i} - \floor {nt}B_{n}\Bigg), \qquad t \in %[0,1],
%\ee
%where the normalizing sequence $(a_n)$ satisfies~\eqref{eq:onedimregvar} and
%$$ B_{n} = \left\{ \begin{array}{cc}
%                                   0, & \quad \alpha \in (0,1], \\
%                                   C \mathrm{E}(Z_{1}), & \quad \alpha \in (1,2).
%                                 \end{array}\right.$$

Let $(X_{t})$ be a moving average process defined by
$$ X_{t} = \sum_{i=0}^{q}C_{i}Z_{t-i}, \qquad t \in \mathbb{Z},$$
and let the corresponding partial sum process be
\be\label{eq:defVn}
V_{n}(t) = \frac{1}{a_{n}}  \sum_{i=1}^{\floor {nt}}X_{i}, \qquad t \in [0,1],
\ee
where the normalizing sequence $(a_n)$ satisfies~\eqref{eq:niz}. Let $B(t) = C$ for $t \in [0,1]$.

\begin{thm}\label{t:FinMA}
Let $(Z_{i})_{i \in \mathbb{Z}}$ be an i.i.d.~sequence of regularly varying random variables with index $\alpha \in (0,2)$, such that $(\ref{e:oceknula})$ and $(\ref{e:sim})$ hold.
Assume $C_{0}, C_{1}, \ldots , C_{q}$ are random variables, independent of $(Z_{i})$, that satisfy
$(\ref{eq:FiniteMAcond})$. Then
$$ V_{n}(\,\cdot\,) \dto \widetilde{B}(\,\cdot\,) V(\,\cdot\,), \qquad n \to \infty,$$
in $D[0,1]$ endowed with the $M_{2}$ topology, where $V$ is an $\alpha$--stable L\'{e}vy process with characteristic triple $(0, \mu, b)$, with $\mu$ as in $(\ref{eq:mu})$ and
$$ b = \left\{ \begin{array}{cc}
                                   0, & \quad \alpha = 1,\\[0.4em]
                                   (p-r)\frac{\alpha}{1-\alpha}, & \quad \alpha \in (0,1) \cup (1,2),
                                 \end{array}\right.$$
and $\widetilde{B}$ is a random element in $D[0,1]$, independent of $V$, such that $\widetilde{B} \eind B$.
\end{thm}

%\begin{rem}\label{r:chartriple}
% The characteristic
% triple of the limiting process $V$ in the above theorem is of the form $(0,\mu,b)$, with $\mu$ as in %$(\ref{eq:mu})$ and
%$$ b = \left\{ \begin{array}{cc}
%                                   0, & \quad \alpha = 1,\\[0.4em]
%                                   (p-r)\frac{\alpha}{1-\alpha}, & \quad \alpha \in (0,1) \cup (1,2).
%                                 \end{array}\right.$$
%\end{rem}

As in Basrak and Krizmani\'{c}~\cite{BaKr} one can prove the following lemma.

\begin{lem}\label{l:first}
\begin{itemize}
  \item[(i)] For $k < q$ it holds
  \begin{eqnarray*}
  \sum_{i=1}^{k}\frac{C\,Z_{i}}{a_{n}} - \sum_{i=1}^{k}\frac{X_{i}}{a_{n}} & = & \sum_{u=0}^{k-1}\frac{Z_{k-u}}{a_{n}} \sum_{s=u+1}^{q}C_{s} - \sum_{u=k-q}^{q-1}\frac{Z_{-u}}{a_{n}} \sum_{s=u+1}^{q}C_{s}\\[0.6em]
  & & - \sum_{u=0}^{q-k-1}\frac{Z_{-u}}{a_{n}} \sum_{s=u+1}^{u+k}C_{s}.
  \end{eqnarray*}
  \item[(ii)] For $k \ge q$ it holds
  \begin{eqnarray*}
  \sum_{i=1}^{k}\frac{C\,Z_{i}}{a_{n}} - \sum_{i=1}^{k}\frac{X_{i}}{a_{n}} & = & \sum_{u=0}^{q-1}\frac{Z_{k-u}}{a_{n}} \sum_{s=u+1}^{q}C_{s} - \sum_{u=0}^{q-1}\frac{Z_{-u}}{a_{n}} \sum_{s=u+1}^{q}C_{s}\\[0.7em]
  & =: & H_{n}(k) - G_{n}.
  \end{eqnarray*}
  \item[(iii)] For $q \le k \le n-q$ it holds
  \begin{eqnarray*}
  \sum_{i=1}^{k}\frac{C\,Z_{i}}{a_{n}} - \sum_{i=1}^{k+q}\frac{X_{i}}{a_{n}} & = & - \sum_{u=0}^{q-1}\frac{Z_{-u}}{a_{n}} \sum_{s=u+1}^{q}C_{s} - \sum_{u=1}^{q}\frac{Z_{k+u}}{a_{n}} \sum_{s=0}^{q-u}C_{s}\\[0.7em]
  & =: & -G_{n} - T_{n}(k).
  \end{eqnarray*}
\end{itemize}
\end{lem}

%\begin{rem}
%Note that random variables $H_{n}(k)$ and $T_{n}(k)$ are independent. (OVO SADA NECE VRIJEDITI!!)
%\end{rem}

\begin{proof} ({\it Theorem~\ref{t:FinMA}})
Since the random variables $Z_{i}$ are i.i.d.~and regularly varying, it is known that
\begin{equation*}
  \sum_{i=1}^{\floor {nt} }\frac{Z_{i}}{a_{n}} - \floor{nt} \mathrm{E} \Big(\frac{Z_{1}}{a_{n}} 1_{\{|Z_{1}| \leq a_{n}\}} \Big), \qquad t \in [0,1],
\end{equation*}
converges in distribution, as $n \to \infty$, in $D[0,1]$ with the $M_{1}$ topology to an $\alpha$--stable L\'{e}vy process with characteristic triple $(0,\mu,0)$ (see Theorem 3.4 in Basrak et al.~\cite{BKS}). By Karamata's theorem, as $n \to \infty$,
\begin{eqnarray*}
  n\,\mathrm{E} \Big( \frac{Z_{1}}{a_{n}} 1_{\{ |Z_{1}| \leq a_{n} \}} \Big) \to (p-r)\frac{\alpha}{1-\alpha}, && \textrm{if}  \ \ \alpha < 1,\\[0.5em]
  n\,\mathrm{E} \Big( \frac{Z_{1}}{a_{n}} 1_{\{ |Z_{1}| > a_{n} \}} \Big) \to (p-r)\frac{\alpha}{\alpha-1}, && \textrm{if} \ \ \alpha >1,
\end{eqnarray*}
with $p$ and $r$ as in (\ref{eq:pq}). Therefore conditions (\ref{e:oceknula}) and (\ref{e:sim}), Corollary 12.7.1 in Whitt~\cite{Whitt02} (which gives a sufficient condition for addition to be continuous in the $M_{1}$ topology) and the continuous mapping theorem yield that
$V_{n}^{Z}(\,\cdot\,) \dto V(\,\cdot\,)$, as $n \to \infty$, in $D[0,1]$ with the $M_{1}$ topology, where
$$ V_{n}^{Z}(t) := \sum_{i=1}^{\floor {nt}}\frac{Z_{i}}{a_{n}}, \qquad t \in [0,1],$$
and
 $V$ is an $\alpha$--stable L\'{e}vy process
 with characteristic triple
$(0,\mu,0)$ if $\alpha=1$ and $(0,\mu,(p-r)\alpha/(1-\alpha))$ if $\alpha \in (0,1) \cup (1,2)$.

It is well known that the space $D[0,1]$ equipped with the Skorohod $J_{1}$ topology is a Polish space (i.e.~metrizable as a complete separable metric space), see Billingsley~\cite{Bi68}, Section 14. The same holds for the $M_{1}$ topology, since it is topologically complete (see Whitt~\cite{Whitt02}, Section 12.8) and separability remains preserved in the weaker topology. Therefore by Corolarry 5.18 in Kallenberg~\cite{Ka97}, we can find a random element $\widetilde{B}$ in $D[0,1]$, independent of $V$, such that $\widetilde{B} \eind B$. This and the fact that $C$ is independent of $V_{n}^{Z}$, by an application of Theorem 3.29 in Kallenberg~\cite{Ka97}, imply
  \begin{equation}\label{e:zajedkonvK}
   (B(\,\cdot\,), V_{n}^{Z}(\,\cdot\,)) \dto (\widetilde{B}(\,\cdot\,), V(\,\cdot\,)), \qquad \textrm{as} \ n \to \infty,
  \end{equation}
  in $D([0,1], \mathbb{R}^{2})$ with the product $M_{1}$ topology.

Let $g \colon D([0,1], \mathbb{R}^{2}) \to D[0,1]$ be a function defined by
$$ g(x) = x_{1}x_{2}, \qquad x=(x_{1},x_{2}) \in D([0,1], \mathbb{R}^{2}),$$
where $(x_{1}x_{2})(t) = x_{1}(t) x_{2}(t)$ for $t \in [0,1]$. Let
$$D_{1} = \{ u \in D([0,1] : \textrm{Disc}(u) = \emptyset \},$$
and
$$ D_{2} = \{ (u,v) \in D([0,1], \mathbb{R}^{2}) : \textrm{Disc}(u) = \emptyset \},$$
where $\textrm{Disc}(u)$ is the set of discontinuity points of $u$. Then by Theorem 13.3.2 in Whitt~\cite{Whitt02} the function $g$ is continuous on the set $D_{2}$ (with the Skorohod $M_{1}$ topology on $D[0,1]$ and product $M_{1}$ topology on $D([0,1], \mathbb{R}^{2})$). Hence $\textrm{Disc}(g) \subseteq D_{2}^{c}$, and
$$ \Pr[ (\widetilde{B}, V) \in \textrm{Disc}(g) ] \leq \Pr [ (\widetilde{B}, V)  \in D_{2}^{c}] \leq \Pr ( \widetilde{B} \in D_{1}^{c}) =\Pr ( B \in D_{1}^{c})=0 .$$
 This allows us to apply the continuous mapping theorem (see for instance Theorem 3.1 in Resnick~\cite{Resnick07}) to relation (\ref{e:zajedkonvK}) which yields
$g( B, V_{n}^{Z}) \dto g( \widetilde{B}, V)$, i.e.
$$ C V_{n}^{Z}(\,\cdot\,) \dto \widetilde{B}(\,\cdot\,) V(\,\cdot\,), \qquad \textrm{as} \ n \to \infty,$$ in $D[0,1]$ with the $M_{1}$ topology.
Using the fact that $M_{1}$ convergence implies $M_{2}$ convergence, we obtain
\be
CV_{n}^{Z}(\,\cdot\,) \dto \widetilde{B}(\,\cdot\,) V(\,\cdot\,), \qquad \textrm{as} \ n \to \infty,
\ee
in $(D[0,1], d_{M_{2}})$ as well. If we can show that for every $\epsilon >0$
$$ \lim_{n \to \infty}\Pr[d_{M_{2}}(CV_{n}^{Z}, V_{n})> \epsilon]=0,$$
an application of  Slutsky's theorem (see for instance Theorem 3.4 in Resnick~\cite{Resnick07})
will imply
$V_{n}(\,\cdot\,) \dto \widetilde{B}(\,\cdot\,) V(\,\cdot\,)$, as $n \to \infty$, in $(D[0,1], d_{M_{2}})$.

Fix $\epsilon >0$ and let $n \in \mathbb{N}$ be large enough, i.e.
 $n > \max\{2q, 2q/\epsilon \}$.
By the definition of the metric $d_{M_{2}}$ we have
\begin{eqnarray*}
  d_{M_{2}}(CV_{n}^{Z},V_{n}) &=& \bigg(\sup_{a \in \Gamma_{CV_{n}^{Z}}} \inf_{b \in \Gamma_{V_{n}}} d(a,b) \bigg) \vee \bigg(\sup_{a \in \Gamma_{V_{n}}} \inf_{b \in \Gamma_{CV_{n}^{Z}}} d(a,b) \bigg) \\[0.4em]
   &= :& Y_{n} \vee T_{n},
\end{eqnarray*}
and therefore
\be\label{eq:AB}
\Pr [d_{M_{2}}(V_{n}^{Z}, V_{n})> \epsilon ] \leq \Pr(Y_{n}>\epsilon) + \Pr(T_{n}>\epsilon)\,.
\ee
In order to estimate the first term on the right hand side of (\ref{eq:AB})
%By the definition of $Y_{n}$, the Hausdorff metric and the choice of number $n$,
note that
\begin{eqnarray}\label{eq:Yn}
  \nonumber\{Y_{n} > \epsilon\} & \subseteq & \{\exists\,a \in \Gamma_{CV_{n}^{Z}} \ \textrm{such that} \ d(a,b) > \epsilon \ \textrm{for every} \ b \in \Gamma_{V_{n}} \} \\[0.6em]
  \nonumber & \subseteq & \{\exists\,k \in \{1,\ldots,q-1\} \ \textrm{such that} \ | CV_{n}^{Z}(k/n) - V_{n}(k/n)| > \epsilon \}\\[0.6em]
  \nonumber & & \cup \ \{\exists\,k \in \{q,\ldots,n-q\} \ \textrm{such that} \ | CV_{n}^{Z}(k/n) - V_{n}(k/n)| > \epsilon\\[0.6em]
  \nonumber & & \hspace*{1.5em} \textrm{and} \  | CV_{n}^{Z}(k/n) - V_{n}((k+q)/n)| > \epsilon \}\\[0.6em]
  \nonumber & & \cup \ \{\exists\,k \in \{n-q+1,\ldots,n\} \ \textrm{such that} \ | CV_{n}^{Z}(k/n) - V_{n}(k/n)| > \epsilon \}\\[0.6em]
  & =: & A^{Y}_{n} \cup B^{Y}_{n} \cup C^{Y}_{n},
\end{eqnarray}
 where the second inclusion above follows from the fact that the paths of $V_n$ and $CV_n^Z$ are constant
on the intervals of the form
\[
 \left[ \frac jn , \frac{j+1}{n} \right)\,,\quad j = 0,1,\ldots, n-1\,.
\]
More precisely, if there is a point $a=(t_a,x_a) \in \Gamma_{CV_{n}^{Z}}$
 such that $d(a, \Gamma_{V_{n}}) > \epsilon$,
  then
 necessarily $t_a \in [i/n, (i+1)/n)$ for some $ i = 1,\ldots, n$.
If $a$ lies on a horizontal
part of the completed graph, then  $x_a = CV_n^Z(i/n)$ and
 $$
 \left| CV_n^Z ( i /n)- V_n ( i /n) \right| \geq d(a, \Gamma_{V_{n}}) >\epsilon.
 $$
Alternatively, if $a$ lies on a
vertical part of the completed graph, then $x_a  \in [CV_n^Z((i-1)/n), CV_n^Z(i/n))$, and one can similarly conclude that
% assuming without loss of generality that
%$V_n^Z((i-1)/n) \leq  V_n^Z(i/n)$ and $V_n((i-1)/n) \leq  V_n(i/n)$, then
%\begin{eqnarray*}
% \lefteqn{\max \left\{
% \left| V_n^Z \left( \frac in \right)- V_n \left( \frac in \right) \right|,
% \left| V_n^Z \left( \frac{i-1}{n} \right)- V_n \left( \frac{i-1}{n} \right) \right|
% \right\}} \qquad \qquad \\[0.5em]
%& \geq &
%\widetilde{d}( x_a,  [V_n ((i-1)/n), V_n(i/n)))
%   \geq d(a, \Gamma_{V_{n}}) >\epsilon\,,
% \end{eqnarray*}
% where $\widetilde{d}$ is the Euclidean metric on $\mathbb{R}$.
%Therefore, for some $ k = 1,\ldots, n$ %($k=i$ or $i-1$)
 \[
 \left| CV_n^Z ( k /n)- V_n ( k /n) \right|>\epsilon\,
 \]
 for some $k = 1, \ldots, n$ (in fact $k=i$ or $k=i-1$; see Basrak and Krizmani\'{c}~\cite{BaKr} for details).
Moreover, if $q \leq  k \leq n-q$, from $q/n< \epsilon/2$ it follows similarly that
  \[
 \left| CV_n^Z (k /n)- V_n (( k+q) /n) \right|>\epsilon\,.
 \]
By Lemma~\ref{l:first} (i) we obtain
\begin{eqnarray}\label{eq:Bnlemmafirst}
  \nonumber \Pr (A^{Y}_{n}) & \leq &  \sum_{k=1}^{q-1} \Pr \Big( \Big| \sum_{i=1}^{k}\frac{C Z_{i}}{a_{n}} - \sum_{i=1}^{k}\frac{X_{i}}{a_{n}} \Big| >\epsilon \Big) \\[0.6em]
    \nonumber  & \leq & \sum_{k=1}^{q-1} \bigg[ \Pr \Big( \sum_{u=0}^{k-1}\frac{|Z_{k-u}|}{a_{n}}
         \sum_{s=u+1}^{q}|C_{s}| > \frac{\epsilon}{3} \Big) + \Pr \Big( \sum_{u=k-q}^{q-1}\frac{|Z_{-u}|}{a_{n}} \sum_{s=u+1}^{q}|C_{s}| > \frac{\epsilon}{3} \Big)\\[0.6em]
  \nonumber  & & \hspace*{2em} + \Pr \Big( \sum_{u=0}^{q-k-1}\frac{|Z_{-u}|}{a_{n}} \sum_{s=u+1}^{u+k}|C_{s}| > \frac{\epsilon}{3} \Big) \bigg]\\[0.6em]
      & \leq &  3(q-1)(2q-1) \Pr \Big( \frac{|Z_{0}|}{a_{n}}\,C_{*} > \frac{\epsilon}{3(2q-1)} \Big),
  \end{eqnarray}
   where $C_{*} = \sum_{s=0}^{q}|C_{s}|$. For an arbitrary $M>0$ it holds that
   \begin{eqnarray*}
     \Pr \Big( \frac{|Z_{0}|}{a_{n}}\,C_{*} > \frac{\epsilon}{3(2q-1)} \Big) &&  \\[0.7em]
      &\hspace*{-16em} =&  \hspace*{-8em} \Pr \Big( \frac{|Z_{0}|}{a_{n}}\,C_{*} > \frac{\epsilon}{3(2q-1)},\,C_{*} > M \Big) + \Pr \Big( \frac{|Z_{0}|}{a_{n}}\,C_{*} > \frac{\epsilon}{3(2q-1)},\,C_{*} \leq M \Big)\\[0.7em]
       &\hspace*{-16em} \leq &  \hspace*{-8em} \Pr \Big( C_{*} > M \Big) + \Pr \Big( \frac{|Z_{0}|}{a_{n}} > \frac{\epsilon}{3(2q-1)M} \Big).
   \end{eqnarray*}
%Using (\ref{e:momcond}), Markov's inequality and the following inequality
%\begin{equation}\label{e:ineq}
%\Big| \sum_{i=1}^{n}a_{i} \Big|^{\gamma} \leq \sum_{i=1}^{n}|a_{i}|^{\gamma},
%\end{equation}
%for real numbers $a_{1}, \ldots, a_{n}$ and $\gamma \in (0,1]$, we obtain
%$$ \Pr \Big(C_{*} > M \Big) \leq \frac{\mathrm{E}  C_{*}^{\delta}}{M^{\delta}} \leq \frac{1}{M^{\delta}} %\sum_{s=0}^{q} \mathrm{E} |C_{s}|^{\delta} \to 0 \quad \textrm{as} \ M \to \infty.$$
%::::::: OVO MOZDA NIJE NI POTREBNO, JER JE JASNO DA $P(C_{*} >M) \to 0$ KAD $M \to \infty$. STOGA JE %MOZDA UVJET (\ref{e:momcond}) NEPOTREBAN, VEC JE DOVOLJAN UVJET PRIJE NJEGA?!:::::::::::::
By the
 regular variation property
we observe
\begin{equation*}
\lim_{n \to \infty} \Pr \Big( \frac{|Z_{0}|}{a_{n}} > \frac{\epsilon}{3(2q-1)M} \Big) =0,
\end{equation*}
and hence from (\ref{eq:Bnlemmafirst}) we get
$$ \limsup_{n \to \infty} \Pr (A^{Y}_{n}) \leq \Pr \Big(C_{*} > M \Big).$$
Letting $M \to \infty$ we conclude
 \be\label{eq:setBn1}
 \lim_{n \to \infty} \Pr (A^{Y}_{n}) = 0.
 \ee
Next, using Lemma~\ref{l:first} (ii) and (iii), for an arbitrary $M>0$
%together with the fact that $H_{n}(k)$ and $T_{n}(k)$ are conditionally independent given $(C_{j})$,
we obtain
\begin{eqnarray*}
  \Pr (B^{Y}_{n}\cap \{ C_{*} \leq M \}) & = &  \Pr \Big( \exists\,k \in \{q,\ldots,n-q\} \ \textrm{such that} \ |H_{n}(k)-G_{n}| > \epsilon \\[0.6em]
  & & \hspace*{2em} \textrm{and} \ |-G_{n}-T_{n}(k)| > \epsilon,\,C_{*} \leq M \Big)\\[0.6em]
   & \hspace*{-16em} \leq & \hspace*{-8em} \Pr \Big( |G_{n}|> \frac{\epsilon}{2},\,C_{*} \leq M \Big) + \sum_{k=q}^{n-q} \Pr \Big( |H_{n}(k)| >
       \frac{\epsilon}{2} \ \textrm{and} \ |T_{n}(k)| > \frac{\epsilon}{2},\,C_{*} \leq M \Big)
\end{eqnarray*}
Note that
\begin{eqnarray*}
  \Pr \Big( |G_{n}|> \frac{\epsilon}{2},\,C_{*} \leq M \Big) & \leq & \Pr \Big( C_{*} \sum_{u=0}^{q-1}\frac{|Z_{-u}|}{a_{n}}> \frac{\epsilon}{2},\,C_{*} \leq M \Big)  \\[0.5em]
   & \leq &  \Pr \Big( \sum_{u=0}^{q-1}\frac{|Z_{-u}|}{a_{n}}> \frac{\epsilon}{2M} \Big)\\[0.5em]
   & \leq &  q \Pr \Big( \frac{|Z_{0}|}{a_{n}} > \frac{\epsilon}{2qM} \Big)
\end{eqnarray*}
Similarly
\begin{eqnarray*}
  \Pr \Big( |H_{n}(k)| >
       \frac{\epsilon}{2} \ \textrm{and} \ |T_{n}(k)| > \frac{\epsilon}{2},\,C_{*} \leq M \Big) & & \\[0.6em]
   & \hspace*{-22em} \leq & \hspace*{-11em} \Pr \Big( \sum_{u=0}^{q-1}\frac{|Z_{k-u}|}{a_{n}} >
       \frac{\epsilon}{2M} \ \textrm{and} \ \sum_{u=1}^{q}\frac{|Z_{k+u}|}{a_{n}} > \frac{\epsilon}{2M}\Big)\\[0.6em]
   & \hspace*{-22em} = & \hspace*{-11em} \Pr \Big( \sum_{u=0}^{q-1}\frac{|Z_{k-u}|}{a_{n}} >
       \frac{\epsilon}{2M} \Big) \Pr \Big(\sum_{u=1}^{q}\frac{|Z_{k+u}|}{a_{n}} > \frac{\epsilon}{2M}\Big)\\[0.6em]
   & \hspace*{-22em} \leq & \hspace*{-11em} \Big[ q \Pr \Big(\frac{|Z_{0}|}{a_{n}} >
       \frac{\epsilon}{2qM} \Big) \Big]^{2},
\end{eqnarray*}
where the equality above holds since the random variables $Z_{i}$ are independent. Therefore
\begin{eqnarray*}\label{e:RVdet1}
 \nonumber \Pr (B^{Y}_{n}\cap \{ C_{*} \leq M \}) & \leq & q \Pr \Big( \frac{|Z_{0}|}{a_{n}} > \frac{\epsilon}{2qM} \Big) + \sum_{k=q}^{n-q} \Big[ q \Pr \Big(\frac{|Z_{0}|}{a_{n}} >
       \frac{\epsilon}{2qM} \Big) \Big]^{2}\\[0.6em]
   &\leq &  q \Pr \Big( \frac{|Z_{0}|}{a_{n}} > \frac{\epsilon}{2qM} \Big) + \frac{q^{2}}{n} \Big[ n \Pr
            \Big( \frac{|Z_{0}|}{a_{n}} > \frac{\epsilon}{2qM} \Big) \Big]^{2}
  \end{eqnarray*}
  and an application of the regular variation property yields
 $$ \lim_{n \to \infty}  \Pr (B^{Y}_{n}\cap \{ C_{*} \leq M \})=0.$$
 Thus
 $$ \limsup_{n \to \infty} \Pr (B^{Y}_{n}) \leq \limsup_{n \to \infty} \Pr (B^{Y}_{n}\cap \{ C_{*} > M \} \leq \Pr (C_{*} > M),$$
and letting again $M \to \infty$ we conclude
 \be\label{eq:setBn3}
 \lim_{n \to \infty} \Pr (B^{Y}_{n}) = 0.
 \ee
In a similar manner as in (\ref{eq:Bnlemmafirst}), but using (ii) from Lemma~\ref{l:first} instead of (i) we get
  \be\label{eq:setBn4}
   \lim_{n \to \infty} \Pr (C^{Y}_{n})=0.
  \ee
  From relations (\ref{eq:Yn}), (\ref{eq:setBn1}), (\ref{eq:setBn3}) and (\ref{eq:setBn4}) we obtain
  \be\label{eq:Ynend}
  \lim_{n \to \infty} \Pr(Y_{n} > \epsilon ) =0.
  \ee

 It remains to estimate the second term on the right hand side of (\ref{eq:AB}).
  For each $k \geq q$, set $V^{Z,\min}_k = \min\{CV^Z_n((k-q)/n), CV^Z_n(k/n) \}$
and $V^{Z,\max}_k = \max\{CV^Z_n((k-q)/n), CV^Z_n(k/n) \}$.
 From the definition of $T_{n}$, the Hausdorff metric and the number $n$ it follows
\begin{eqnarray}\label{eq:Zn}
  \nonumber\{T_{n} > \epsilon\} & \subseteq & \{\exists\,a \in \Gamma_{V_{n}} \ \textrm{such that} \ d(a,b) > \epsilon \ \textrm{for every} \ b \in \Gamma_{CV_{n}^{Z}} \} \\[0.6em]
  \nonumber & \subseteq & \{\exists\,k \in \{1,\ldots,2q-1\} \
\ \textrm{such that} \ | V_{n}(k/n) - CV_{n}^{Z}(k/n)| > \epsilon \}\\[0.6em]
  \nonumber & & \cup \ \Big\{\exists\,k \in \{2q,\ldots,n\} \
   \textrm{such that} \
  \widetilde{d}(V_{n}(k/n), [V_{k}^{Z, \min}, V_{k}^{Z, \max}]) > \epsilon
    \Big\}\\[0.6em]
  & =: & A^{T}_{n} \cup B^{T}_{n},
\end{eqnarray}
where $\widetilde{d}$ is the Euclidean metric on $\mathbb{R}$.
 The argument behind the second inclusion in (\ref{eq:Zn}) is similar to the one given after (\ref{eq:Yn}).
%It again, relies on the fact that the paths of $V_n$ and $V_n^Z$ are constant
%on the intervals of the form
%$ \left[ j/n, {(j+1)}/{n} \right),\  j = 0,1,\ldots, n-1\,.$
 Indeed, assume there is a point $a=(t_a,x_a) \in \Gamma_{V_{n}}$
 such that
 \begin{equation}\label{eq:novo1}
 d(a, \Gamma_{CV_{n}^{Z}}) > \epsilon.
 \end{equation}
  Then
 necessarily $t_a \in [i/n, (i+1)/n)$ for some $ i = 1,\ldots, n$.
 The case $ i \leq 2q-1$ is covered  by the same argument used to obtain (\ref{eq:Yn}) and the set $A^{Y}_{n}$.
 Therefore, we may assume $i \geq 2q$. From (\ref{eq:novo1}) we immediately obtain
 \begin{equation}\label{eq:novo2}
 d ( a, ( i/n, CV_{n}^{Z}(i/n))) > \epsilon \quad \textrm{and} \quad d ( a, ( (i-q)/n, CV_{n}^{Z}((i-q)/n))) > \epsilon.
 \end{equation}
 Suppose first that $x_a = V_n(i/n)$ for some $i=2q,\ldots,n$. Recall that $q/n < \epsilon/2$. Since
 $ \max \{|t_{a}-i/n|, |t_{a}- (i-q)/n| \} \leq (q+1)/n < \epsilon$, from (\ref{eq:novo2}) we conclude that
 $$ \widetilde{d}(V_{n}(i/n), [V_{i}^{Z, \min}, V_{i}^{Z, \max}]) > \epsilon.$$
 If $x_a  \in [V_n((i-1)/n), V_n(i/n))$ (in this case $t_{a}=i/n$), relation (\ref{eq:novo2}) again implies
 $ \widetilde{d}(x_{a}, [V_{i}^{Z, \min}, V_{i}^{Z, \max}]) > \epsilon$,
 and similarly
 $ \widetilde{d}(x_{a}, [V_{i-1}^{Z, \min}, V_{i-1}^{Z, \max}]) > \epsilon$. Thus we obtain
 $$ \max \{ \widetilde{d}(V_{n}(i/n), [V_{i}^{Z, \min}, V_{i}^{Z, \max}]), \widetilde{d}(V_{n}((i-1)/n), [V_{i-1}^{Z, \min}, V_{i-1}^{Z, \max}])\} > \epsilon.$$
 Finally we conclude that there exists $k \in \{2q,\ldots,n\}$
   such that
  $$\widetilde{d}(V_{n}(k/n), [V_{k}^{Z, \min}, V_{k}^{Z, \max}]) > \epsilon.$$

Using Lemma~\ref{l:first} (i) and (ii), one could similarly as before for the set $A^{Y}_{n}$ obtain
\be\label{eq:Cnfirst}
\lim_{n \to \infty} \Pr( A^{T}_{n})=0.
\ee
 Note that $\Pr (B_n^T)$ is bounded above by
\begin{eqnarray*}
\lefteqn{ \Pr \left(
\exists\,k \in \{2q,\ldots,n\} \ \textrm{such that} \
\sum_{i=1}^k \frac{X_i}{a_n} > V^{Z,\max}_k + \epsilon \right)}\\
&+&
\Pr \left(
\exists\,k \in \{2q,\ldots,n\} \ \textrm{such that} \
\sum_{i=1}^k \frac{X_i}{a_n} < V^{Z,\min}_k - \epsilon \right)\,.
%=:I_1^n+I^2_n
\end{eqnarray*}
In the sequel we consider only the first of these two probabilities,
since the other one can be handled in a similar manner.
The first probability
using Lemma~\ref{l:first} can be bounded  by
\begin{eqnarray*}
\lefteqn{\Pr \left(
\exists\,k \in \{2q,\ldots,n\} \ \textrm{such that} \
G_n - H_n(k) > \epsilon \ \mbox{ and }\
G_n + T_n(k-q) > \epsilon   \right)} \\
& \leq &
\Pr \left(G_n  >  \frac{\epsilon}{2}\right)\\
&&  +
\Pr \left(
\exists\,k \in \{2q,\ldots,n\} \ \textrm{such that} \
H_n(k) < -  \frac{\epsilon}{2} \ \mbox{ and }\
T_n(k-q) >  \frac{\epsilon}{2}   \right)\,.
\end{eqnarray*}
From the calculations yielding (\ref{eq:setBn3}) we conclude that $\Pr (G_n  > \epsilon/2 )\to 0$
 as $n \to \infty$. The second term is bounded by
\begin{equation}\label{eq:BnTpom}
\Pr(C_{*}> M) + \sum_{k=2q}^{n} \Pr \left(
%\exists\,k \in \{2q,\ldots,n\} \ \textrm{such that} \
H_{n}(k) < - \frac{\epsilon}{2}
\ \mbox{ and }
T_{n}(k-q) > \frac{\epsilon}{2},\,C_{*} \leq M
 \right)
\end{equation}
for an arbitrary $M>0$. Note that
$$
 H_n(k)= \sum_{u=0}^{q-1}\frac{Z_{k-u}}{a_{n}} \sum_{s=u+1}^{q}C_{s}
\ \mbox{ and } T_n(k-q) = \sum_{u=0}^{q-1}\frac{Z_{k-u}}{a_{n}}
 \sum_{s=0}^{u}C_{s}.
 $$
Therefore for a fixed $k \in \{2q,\ldots,n\}$, on the event
$ \{ H_{n}(k)
< - \epsilon/2  \ \textrm{and} \
T_{n}(k-q)
> \epsilon/2,\,C_{*} \leq M \}$
 there exist $i,j \in \{0,\ldots,q-1\}$ such that
 $$\frac{Z_{k-i}}{a_{n}} \sum_{s=i+1}^{q}C_{s}
< - \frac{\epsilon}{2q} \quad \textrm{and} \quad \frac{Z_{k-j}}{a_{n}}
 \sum_{s=0}^{j}C_{s}
> \frac{\epsilon}{2q}.$$
From \eqref{eq:FiniteMAcond} it follows that the sums
$\sum_{s=0}^{j}C_{s}$
and $ \sum_{s=i+1}^{q}C_{s}$ are a.s.~of the same sign and
their absolute values are bounded by $C_{*}$.
Hence if these sums are positive we obtain
$Z_{k-i}M/a_{n} < - \epsilon/(2q)$ and
$Z_{k-j}M/a_{n} > \epsilon/(2q)$, while if they are negative we obtain $Z_{k-i}M/a_{n} > \epsilon/(2q)$ and $Z_{k-j}M/a_{n} < -\epsilon/(2q)$. Note that the case $i=j$ is not possible since then we would have $Z_{k-i}<0$ and $Z_{k-i}>0$. From this, using the stationarity of the sequence $(Z_{i})$, we conclude that the expression in (\ref{eq:BnTpom}) is bounded by
\begin{eqnarray*}\label{e:RVdet2}
\nonumber & & \hspace*{-2em} {\Pr(C_{*}> M) +  n \Pr\left(
\exists\, i,j \in \{0,\ldots,q-1\},\,i \neq j \ \textrm{s.t.} \
M \frac{Z_{-i }}{a_{n}}< -  \frac{\epsilon}{2q}
 \mbox{ and } M \frac{Z_{-j }}{a_{n}} >  \frac{\epsilon}{2q} \right)} \\
& \leq & \Pr(C_{*}> M) + n {q \choose 2}
 \bigg[ \Pr\left(
 \frac{|Z_{0}|}{a_{n}} >   \frac{\epsilon}{2 q M}
 \right) \bigg]^2,
\end{eqnarray*}
which tends to 0 if we first let $n\toi$ and then $M \to \infty$.
Together with relations (\ref{eq:Zn}) and (\ref{eq:Cnfirst}) this implies
\be\label{eq:Tnend}
\lim_{n \to \infty} \Pr(T_{n}>\epsilon)=0.
\ee
Now from (\ref{eq:AB}), (\ref{eq:Ynend}) and (\ref{eq:Tnend}) we obtain
\be
\lim_{n \to \infty} \Pr [d_{M_{2}}(CV_{n}^{Z}, V_{n})> \epsilon ]=0,
\ee
and finally we conclude that $V_{n}(\,\cdot\,) \dto \widetilde{B}(\,\cdot\,)V(\,\cdot\,)$, as $n \to \infty$, in $(D[0,1], d_{M_{2}})$.
This concludes the proof.
\end{proof}

\section{Infinite order MA processes}
\label{S:InfiniteMA}

Let $(X_{i})$ be a moving average process defined by
$$ X_{i} = \sum_{j=0}^{\infty}C_{j}Z_{i-j}, \qquad i \in \mathbb{Z},$$
where $(Z_{i})$ is an i.i.d.~sequence of regularly varying random variables with index $\alpha \in (0,2)$, such that $\mathrm{E}Z_{i}=0$ if $\alpha \in (1,2)$ and $Z_{i}$ is symmetric if $\alpha =1$.
Let $\{C_{i}, i=0,1,2,\ldots\}$ be a sequence of random variables, independent of $(Z_{i})$, satisfying
\be\label{eq:convcond}
 \sum_{i=0}^{\infty}\mathrm{E}|C_{i}|^{\delta} < \infty \qquad \textrm{for some} \ \delta < \alpha, \ 0 < \delta \leq 1,
\ee
 and
\be\label{eq:InfiniteMAcond}
0 \le \sum_{i=0}^{s}C_{i} \Bigg/ \sum_{i=0}^{\infty}C_{i} \le 1 \ \ \textrm{a.s.} \qquad \textrm{for every} \ s=0, 1, 2 \ldots.
\ee
Let $ C = \sum_{i=0}^{\infty}C_{i}$ and $B(t)=C$ for all $t \in [0,1]$. Condition (\ref{eq:convcond}) implies $C$ is a.s.~finite, and ensures that the series in the definition of $X_{i}$ above converges almost surely.
Define further the corresponding partial sum stochastic process $V_{n}$ as in (\ref{eq:defVn}). Beside the above stated conditions, we require also the following conditions: for $\alpha \in (0,1)$
\be\label{e:mod1}
 \sum_{i=0}^{\infty}\mathrm{E}|C_{i}|^{\gamma} < \infty \qquad \textrm{for some} \ \gamma \in (\alpha, 1),
\ee
and for $\alpha \in [1,2)$
 \begin{equation}\label{e:vod1}
 \lim_{n \to \infty} (\ln n)^{1+\eta}\,\mathrm{E} \bigg[ \bigg( \sum_{i=n}^{\infty} |C_{i}| \bigg)^{\eta-\delta} \sum_{j=n}^{\infty}|C_{j}|^{\delta} \bigg] =0 \qquad \textrm{for some} \ \eta > \alpha.
 \end{equation}
 The latter condition is borrowed from Avram and Taqqu~\cite{AvTa92}, where they studied $M_{1}$ functional convergence of sums of moving averages with deterministic coefficients. Since in the case $\alpha \in (1,2)$ we will also need that the series $\sum_{i=1}^{\infty} \mathrm{E}|C_{i}|$ converges, we assume $\delta=1$ in (\ref{eq:convcond}) if $\alpha >1$.

%\begin{equation}\label{e:tcond}
%\lim_{q \to \infty} \limsup_{n \to \infty} \frac{1}{n} \bigg[ \sum_{i=1}^{n} \mathrm{E} \bigg( %\sum_{j=q+1}^{\infty}|C_{j}| \bigg)^{\gamma} + \sum_{i=-\infty}^{0} \mathrm{E} \bigg( %\sum_{j=q-i+1}^{q-n+1}|C_{j}| \bigg)^{\gamma} \bigg] = 0.
%\end{equation}

\begin{thm}\label{t:InfMA}
Let $(Z_{i})_{i \in \mathbb{Z}}$ be an i.i.d.~sequence of regularly varying random variables with index $\alpha \in (0,2)$. Suppose that conditions $(\ref{e:oceknula})$ and $(\ref{e:sim})$ hold. Let $\{C_{i}, i=0,1,2,\ldots\}$ be a sequence of random variables, independent of $(Z_{i})$, such that $(\ref{eq:convcond})$ and $(\ref{eq:InfiniteMAcond})$ hold. Assume also $(\ref{e:mod1})$ holds if $\alpha \in (0,1)$, and $(\ref{e:vod1})$ if $\alpha \in [1,2)$. Then
$$ V_{n}(\,\cdot\,) \dto \widetilde{B}(\,\cdot\,) V(\,\cdot\,), \qquad n \to \infty,$$
in $D[0,1]$ endowed with the $M_{2}$ topology, where $V$ is an $\alpha$--stable L\'{e}vy process with characteristic triple $(0, \mu, b)$, with $\mu$ as in $(\ref{eq:mu})$ and
$$ b = \left\{ \begin{array}{cc}
                                   0, & \quad \alpha = 1,\\[0.4em]
                                   (p-r)\frac{\alpha}{1-\alpha}, & \quad \alpha \in (0,1) \cup (1,2),
                                 \end{array}\right.$$
   and $\widetilde{B}$ is a random element in $D[0,1]$, independent of $V$, such that $\widetilde{B} \eind B$.
\end{thm}

%\begin{rem}
%The char of the limiting process $V$ in Theorem~\ref{t:InfMA} is of the same form as in %Remark~\ref{r:chartriple}.
%\end{rem}

\begin{proof}
For $q \in \mathbb{N}$ define
$$ X_{i}^{q} = \sum_{j=0}^{q-1}C_{j}Z_{i-j} + C'_{q} Z_{i-q}, \qquad i \in \mathbb{Z},$$
where $C'_{q}= \sum_{i=q}^{\infty}C_{i}$,
and
$$ V_{n, q}(t) = \sum_{i=1}^{\floor{nt}} \frac{X_{i}^{q}}{a_{n}}, \qquad t \in [0,1].$$
Now we treat separately the cases $\alpha \in (0,1)$, $\alpha \in (1,2)$ and $\alpha=1$.\\[-0.8em]

Case $\alpha \in (0,1)$.
Fix $q \in \mathbb{N}$.
 Since the coefficients $C_{0}, \ldots, C_{q-1}, C'_{q}$ satisfy condition (\ref{eq:FiniteMAcond}), an application of Theorem~\ref{t:FinMA} to a finite order moving average process $(X_{i}^{q})_{i}$ yields that, as $n \to \infty$,
\be
V_{n, q}(\,\cdot\,) \dto \widetilde{B}(\,\cdot\,) V(\,\cdot\,)
\ee
in $(D[0,1], d_{M_{2}})$. If we show that for every $\epsilon >0$
$$ \lim_{q \to \infty} \limsup_{n \to \infty}\Pr[d_{M_{2}}(V_{n, q}, V_{n})> \epsilon]=0,$$
then by a generalization of Slutsky's theorem (see for instance Theorem 3.5 in Resnick~\cite{Resnick07}) it will follow $V_{n}(\,\cdot\,) \dto \widetilde{B}(\,\cdot\,) V(\,\cdot\,)$, as $n \to \infty$, in $(D[0,1], d_{M_{2}})$. Since the Skorohod $M_{2}$ metric on $D[0,1]$ is bounded above by the uniform metric on $D[0,1]$, it suffices to show that
$$ \lim_{q \to \infty} \limsup_{n \to \infty}\Pr \bigg( \sup_{0 \leq t \leq 1}|V_{n, q}(t) - V_{n}(t)|> \epsilon \bigg)=0.$$
Recalling the definitions, we have
\begin{eqnarray*}
  \lim_{q \to \infty} \limsup_{n \to \infty}\Pr \bigg( \sup_{0 \leq t \leq 1}|V_{n, q}(t) - V_{n}(t)|> \epsilon \bigg) & &  \\[0.6em]
   & \hspace*{-26em} \leq & \hspace*{-13em} \lim_{q \to \infty} \limsup_{n \to \infty}\Pr \bigg( \sum_{i=1}^{n}\frac{|X_{i}^{q}-X_{i}|}{a_{n}} > \epsilon \bigg).
\end{eqnarray*}
Put $C''_{q} = C'_{q} - C_{q} = \sum_{j=q+1}^{\infty}C_{j}$ and observe
\begin{eqnarray*}
  \sum_{i=1}^{n}|X_{i}^{q}-X_{i}| & = & \sum_{i=1}^{n} \bigg| \sum_{j=0}^{q-1}C_{j}Z_{i-j} + C'_{q}Z_{i-q} - \sum_{j=0}^{\infty}C_{j}Z_{i-j}\bigg| \\[0.6em]
  & = & \sum_{i=1}^{n} \bigg| C''_{q} Z_{i-q} - \sum_{j=q+1}^{\infty}C_{j}Z_{i-j}\bigg|\\[0.6em]
  & \leq & \sum_{i=1}^{n} \bigg[ |C''_{q}|\,|Z_{i-q}| + \sum_{j=q+1}^{\infty}|C_{j}|\,|Z_{i-j}| \bigg]\\[0.6em]
   & \leq & \bigg( 2 \sum_{j=q+1}^{\infty}|C_{j}| \bigg) \sum_{i=1}^{n}|Z_{i-q}| + \sum_{i=-\infty}^{0}|Z_{i-q}| \sum_{j=1}^{n}|C_{q-i+j}|.
\end{eqnarray*}
Let
$$  D^{n,q}_{i} =  \left\{ \begin{array}{cl}
                                   \displaystyle 2 \sum_{j=q+1}^{\infty}|C_{j}|, & \quad i=1,\ldots,n,\\[1.5em]
                                   \displaystyle \sum_{j=1}^{n} |C_{q-i+1}|, & \quad  i \leq 0.
                                 \end{array}\right.$$
Therefore it is enough to show
\be\label{eq:Oiqn1}
\lim_{q \to \infty} \limsup_{n \to \infty} \Pr \bigg(\sum_{i=-\infty}^{n} \frac{D^{n,q}_{i}|Z_{i-q}|}{a_{n}} > \epsilon \bigg)=0.
\ee
Let
$$ Z^{\leq}_{i,n} = \frac{Z_{i}}{a_{n}} 1_{\big\{ \frac{|Z_{i}|}{a_{n}} \leq 1 \big\}} \quad \textrm{and} \quad
Z^{>}_{i,n} = \frac{Z_{i}}{a_{n}} 1_{\big\{ \frac{|Z_{i}|}{a_{n}} > 1 \big\}},$$
and note that the probability in (\ref{eq:Oiqn1}) is bounded above by
\begin{equation}\label{e:Diqn1}
\Pr \bigg(\sum_{i=-\infty}^{n} D^{n,q}_{i} |Z^{\leq}_{i-q,n}| > \frac{\epsilon}{2} \bigg) +
\Pr \bigg(\sum_{i=-\infty}^{n} D^{n,q}_{i} |Z^{>}_{i-q,n}| > \frac{\epsilon}{2} \bigg).
\end{equation}
Using Markov's inequality, the triangle inequality $|\sum_{i=1}^{\infty}a_{i}|^{s} \leq \sum_{i=1}^{\infty}|a_{i}|^{s}$ with $s \in (0,1]$, the fact that $(C_{i})$ is independent of $(Z_{i})$ and the stationarity of the sequence $(Z_{i})$, for the first term in (\ref{e:Diqn1}) we obtain
\begin{eqnarray*}
  \nonumber \Pr \bigg(\sum_{i=-\infty}^{n} D^{n,q}_{i} |Z^{\leq}_{i-q,n}| > \frac{\epsilon}{2} \bigg) & \leq & \Big(\frac{\epsilon}{2} \Big)^{-\gamma} \mathrm{E} \bigg( \sum_{i=-\infty}^{n} D^{n,q}_{i}|Z^{\leq}_{i-q,n}| \bigg)^{\gamma} \\[0.5em]
   & \leq & \Big(\frac{\epsilon}{2} \Big)^{-\gamma} \mathrm{E} \bigg( \sum_{i=-\infty}^{n} (D^{n,q}_{i})^{\gamma}|Z^{\leq}_{i-q,n}|^{\gamma} \bigg)\\[0.5em]
  & \leq & \Big(\frac{\epsilon}{2} \Big)^{-\gamma} \mathrm{E}|Z^{\leq}_{1,n}|^{\gamma}  \sum_{i=-\infty}^{n} \mathrm{E}(D^{n,q}_{i})^{\gamma}.
\end{eqnarray*}
Again by triangle inequality we have
$$ \sum_{i=-\infty}^{n} \mathrm{E} (D^{n,q}_{i})^{\gamma} \leq 2^{\gamma}n  \sum_{j=q+1}^{\infty} \mathrm{E} |C_{j}|^{\gamma} + \sum_{i=-\infty}^{0} \sum_{j=1}^{n} \mathrm{E} |C_{q-i+j}|^{\gamma},$$
Note that every $\mathrm{E}|C_{j}|^{\gamma}$, for $j=q+1, q+2, \ldots$, appears in the sum $\sum_{i=-\infty}^{0} \sum_{j=1}^{n} \mathrm{E} |C_{q-i+j}|^{\gamma}$ at most $n$ times, and hence
\begin{eqnarray}\label{e:Diqn2}
  \nonumber \Pr \bigg(\sum_{i=-\infty}^{n} D^{n,q}_{i} |Z^{\leq}_{i-q,n}| > \frac{\epsilon}{2} \bigg) & \leq & \Big(\frac{\epsilon}{2} \Big)^{-\gamma} \mathrm{E}|Z^{\leq}_{1,n}|^{\gamma}  \bigg(
  2^{\gamma}n  \sum_{j=q+1}^{\infty} \mathrm{E} |C_{j}|^{\gamma} + n  \sum_{j=q+1}^{\infty} \mathrm{E} |C_{j}|^{\gamma} \bigg) \\[0.6em]
  & = & (2^{\gamma}+1) \Big(\frac{\epsilon}{2} \Big)^{-\gamma} n \mathrm{E}|Z^{\leq}_{1,n}|^{\gamma}  \sum_{j=q+1}^{\infty} \mathrm{E} |C_{j}|^{\gamma}.
\end{eqnarray}
Similarly
\begin{equation}\label{e:Diqn3}
   \Pr \bigg(\sum_{i=-\infty}^{n} D^{n,q}_{i} |Z^{>}_{i-q,n}| > \frac{\epsilon}{2} \bigg)
   \leq (2^{\delta}+1) \Big(\frac{\epsilon}{2} \Big)^{-\delta} n \mathrm{E}|Z^{>}_{1,n}|^{\delta}  \sum_{j=q+1}^{\infty} \mathrm{E} |C_{j}|^{\delta}.
\end{equation}
By Karamata's theorem and (\ref{eq:niz}), as $n \to \infty$,
$$ n\mathrm{E}|Z^{\leq}_{1,n}|^{\gamma} = \frac{\mathrm{E}(|Z_{1}|^{\gamma}1_{\{ |Z_{1}| \leq a_{n}\}})}{a_{n}^{\gamma} \Pr(|Z_{1}| > a_{n})} \cdot n \Pr(|Z_{1}| > a_{n}) \to \frac{\alpha}{\gamma - \alpha} < \infty$$
and
$$ n\mathrm{E}|Z^{>}_{1,n}|^{\delta} = \frac{\mathrm{E}(|Z_{1}|^{\delta}1_{\{ |Z_{1}| > a_{n}\}})}{a_{n}^{\delta} \Pr(|Z_{1}| > a_{n})} \cdot n \Pr(|Z_{1}| > a_{n}) \to \frac{\alpha}{\alpha - \delta} < \infty.$$
From this and relations (\ref{e:Diqn2}) and (\ref{e:Diqn3}) we conclude that
$$ \limsup_{n \to \infty} \Pr \bigg(\sum_{i=-\infty}^{n} \frac{D^{n,q}_{i}|Z_{i-q}|}{a_{n}} > \epsilon \bigg) \leq M \bigg( \sum_{j=q+1}^{\infty} \mathrm{E} |C_{j}|^{\gamma} + \sum_{j=q+1}^{\infty} \mathrm{E} |C_{j}|^{\delta} \bigg),$$
where $M =  (2^{\gamma}+1) (\epsilon/2)^{-\gamma} \alpha/(\gamma - \alpha) +  (2^{\delta}+1) (\epsilon/2)^{-\delta} \alpha/(\alpha - \delta) < \infty$. Now letting $q \to \infty$, conditions (\ref{eq:convcond}) and (\ref{e:mod1}) imply (\ref{eq:Oiqn1}), which means that $V_{n}(\,\cdot\,) \dto \widetilde{B}(\,\cdot\,) V(\,\cdot\,)$, as $n \to \infty$, in $(D[0,1], d_{M_{2}})$.\\[-0.8em]

Case $\alpha \in (1,2)$.
Let $(q_{n})$ be a sequence of positive integers such that $q_{n}=\lfloor n^{1/10} \rfloor$. We first show that $\lim_{n \to \infty} \Pr[d_{M_{2}}(V_{n,q_{n}}, V_{n})> \epsilon]=0$ for every $\epsilon >0$. For this, similar to the case $\alpha \in (0,1)$, it suffices to show that
$$ \lim_{n \to \infty}\Pr \bigg( \sup_{0 \leq t \leq 1}|V_{n, q_{n}}(t) - V_{n}(t)|> \epsilon \bigg)=0.$$
Recalling the definitions, we have
\begin{equation*}
  V_{n, q_{n}}(t) - V_{n}(t) = \frac{1}{a_{n}} \sum_{i=1}^{\lfloor nt \rfloor}(X_{i}^{q}-X_{i}) = \frac{1}{a_{n}} \sum_{i=1}^{\lfloor nt \rfloor} \bigg( C''_{q_{n}}Z_{i-q_{n}} + \sum_{j=q_{n}+1}^{\infty}C_{j}Z_{i-j} \bigg),
\end{equation*}
and hence
\begin{eqnarray}\label{e:aproksqn}
  \nonumber\Pr \bigg( \sup_{0 \leq t \leq 1}|V_{n, q_{n}}(t) - V_{n}(t)|> \epsilon \bigg) &&  \\[0.5em]
  \nonumber & \hspace*{-22em} \leq & \hspace*{-11em}  \Pr \bigg( \sup_{0 \leq t \leq 1} \bigg| \sum_{i=1}^{\lfloor nt \rfloor} \frac{C''_{q_{n}} Z_{i-q_{n}}}{a_{n}} \bigg| > \frac{\epsilon}{2} \bigg) + \Pr \bigg( \sup_{0 \leq t \leq 1} \bigg| \sum_{i=1}^{\lfloor nt \rfloor} \sum_{j=q_{n}+1}^{\infty} \frac{C_{j} Z_{i-j}}{a_{n}} \bigg| > \frac{\epsilon}{2} \bigg)\\[0.6em]
   & \hspace*{-22em} =:& \hspace*{-11em} I_{1} + I_{2}.
\end{eqnarray}
Let
$$ \widetilde{Z}^{\leq}_{i,n} = Z^{\leq}_{i,n} - \mathrm{E} Z^{\leq}_{i,n} \qquad \textrm{and} \qquad
\widetilde{Z}^{>}_{i,n} = Z^{>}_{i,n} + \mathrm{E} Z^{\leq}_{i,n},$$
and note that $Z_{i}/a_{n}=  \widetilde{Z}^{\leq}_{i,n} + \widetilde{Z}^{>}_{i,n}$, $ \mathrm{E} \widetilde{Z}^{\leq}_{i,n}=0$ and also
$ \mathrm{E} \widetilde{Z}^{>}_{i,n} = \mathrm{E} Z^{>}_{i,n} + \mathrm{E} Z^{\leq}_{i,n} = \mathrm{E}(Z_{i}/a_{n})=0$. Thus
\begin{eqnarray*}
  I_{1} & \leq & \Pr \bigg( \sup_{1 \leq k \leq n} \bigg| \sum_{i=1}^{k} C''_{q_{n}} \widetilde{Z}^{\leq}_{i-q_{n},n} \bigg| > \frac{\epsilon}{4} \bigg) + \Pr \bigg( \sup_{1 \leq k \leq n} \bigg| \sum_{i=1}^{k} C''_{q_{n}} \widetilde{Z}^{>}_{i-q_{n},n} \bigg| > \frac{\epsilon}{4} \bigg)\\[0.4em]
  &=:& I_{11}+I_{12}.
\end{eqnarray*}
 Since $C''_{q_{n}}$ is independent of $(Z_{i})$ and $\mathrm{E} \widetilde{Z}^{\leq}_{i,n}=0$, it follows that $( \sum_{i=1}^{k} C''_{q_{n}} \widetilde{Z}^{\leq}_{i-q_{n},n} )_{k}$ is a martingale (with respect to the filtration $(\mathcal{F}_{k})$, where the $\sigma$--field $\mathcal{F}_{k}$ is generated by $C_{i},\,i \geq 0$ and $Z_{j-q_{n}},\,j \leq k-q_{n}$). Hence by Markov's inequality and Doob's maximal inequality
 $$ \mathrm{E} \bigg( \sup_{1 \leq k \leq n} |S_{k}| \bigg)^{\kappa} \leq \Big( \frac{\kappa}{\kappa-1} \Big)^{\kappa} \mathrm{E}|S_{n}|^{\kappa},$$
 which holds for $\kappa >1$ and $(S_{k})_{k}$ a martingale (see Durrett~\cite{Du96}, p.~251) we obtain
 $$ I_{11} \leq \Big( \frac{\epsilon}{4} \Big)^{-\eta} \Big( \frac{\eta}{\eta-1} \Big)^{\eta} \mathrm{E} \bigg| \sum_{i=1}^{n} C''_{q_{n}} \widetilde{Z}^{\leq}_{i-q_{n},n} \bigg|^{\eta}.$$
Note that $(C''_{q_{n}}\widetilde{Z}^{\leq}_{i-q_{n},n})_{i}$ is a martingale difference sequence, and hence by the Bahr-Esseen inequality
$$ \mathrm{E} \bigg| \sum_{j=1}^{n}Y_{j} \bigg|^{\kappa} \leq 2 \sum_{j=1}^{n} \mathrm{E}|Y_{j}|^{\kappa},$$
which holds for $ \kappa \in [1,2]$ and $(Y_{j})_{j}$ a martingale-difference sequence
(see Chatterji~\cite{Ch69}, Lemma 1) we have
\begin{eqnarray*}
   I_{11} & \leq & 2 \Big( \frac{\epsilon}{4} \Big)^{-\eta} \Big( \frac{\eta}{\eta-1} \Big)^{\eta} \sum_{i=1}^{n} \mathrm{E} | C''_{q_{n}} \widetilde{Z}^{\leq}_{i-q_{n},n} |^{\eta}\\[0.4em]
    & = & 2 \Big( \frac{\epsilon}{4} \Big)^{-\eta} \Big( \frac{\eta}{\eta-1} \Big)^{\eta} \mathrm{E} | \widetilde{Z}^{\leq}_{1,n} |^{\eta}  \sum_{i=1}^{n} \mathrm{E} | C''_{q_{n}}|^{\eta} \\[0.4em]
   &=& 2 \Big( \frac{\epsilon}{4} \Big)^{-\eta} \Big( \frac{\eta}{\eta-1} \Big)^{\eta} n\mathrm{E} | \widetilde{Z}^{\leq}_{1,n}|^{\eta} \mathrm{E}| C''_{q_{n}}|^{\eta}.
\end{eqnarray*}
 Using the inequality $|a-b|^{\eta} \leq 2^{\eta} (|a|^{\eta} + |b|^{\eta})$ and a special case of Jensen's inequality
 $$ (\mathrm{E}|Y|)^{\kappa} \leq \mathrm{E}|Y|^{\kappa}$$
 (which holds for $\kappa \geq 1$)
  we have
 \begin{equation}\label{e:jensen}
 \mathrm{E}|\widetilde{Z}^{\leq}_{1,n}|^{\eta} \leq 2^{\eta} [ \mathrm{E}|Z^{\leq}_{1,n}|^{\eta} + ( \mathrm{E} |Z^{\leq}_{1,n}| )^{\eta} ] \leq 2^{\eta +1} \mathrm{E}|Z^{\leq}_{1,n}|^{\eta},
\end{equation}
and hence
$$ I_{11} \leq 2^{2-\eta} \epsilon^{-\eta} \Big( \frac{\eta}{\eta-1} \Big)^{\eta} n\mathrm{E} |Z^{\leq}_{1,n}|^{\eta} \mathrm{E}| C''_{q_{n}}|^{\eta}.$$
Note that
\begin{equation}\label{e:AT-tc}
\mathrm{E}| C''_{q_{n}}|^{\eta} = \mathrm{E} (|C''_{q_{n}}|^{\eta-\delta} \cdot |C''_{q_{n}}|^{\delta}) \leq \mathrm{E} \bigg[ \bigg( \sum_{i=q_{n}+1}^{\infty} |C_{i}|
 \bigg)^{\eta-\delta} \sum_{j=q_{n}+1}^{\infty}|C_{j}|^{\delta} \bigg],
\end{equation}
and thus
condition (\ref{e:vod1}) yields $\lim_{n \to \infty}\mathrm{E}| C''_{q_{n}}|^{\eta}=0$. This and the fact that $ \lim_{n \to \infty} n \mathrm{E}|Z^{\leq}_{1,n}|^{\eta} = \alpha/(\eta - \alpha)$ (which holds by Karamata's theorem) allows us to conclude that
$ \lim_{n \to \infty}I_{11}=0$.

For $I_{12}$ by Markov's inequality we obtain
$$ I_{12} \leq \Pr \bigg( \sum_{i=1}^{n} |C''_{q_{n}} \widetilde{Z}^{>}_{i-q_{n},n}| > \frac{\epsilon}{4} \bigg) \leq
\Big( \frac{\epsilon}{4} \Big)^{-1} \mathrm{E}|\widetilde{Z}^{>}_{1,n}| \sum_{i=1}^{n} \mathrm{E} |C''_{q_{n}}|.$$
Since $\widetilde{Z}^{>}_{i,n} = Z^{>}_{i,n} - \mathrm{E} Z^{>}_{i,n}$, it holds that
\begin{equation}\label{e:AT-tc4}
\mathrm{E}|\widetilde{Z}^{>}_{1,n}| \leq  \mathrm{E}|Z^{>}_{1,n}| + |\mathrm{E} Z^{>}_{1,n}| \leq 2 \mathrm{E}|Z^{>}_{1,n}|.
\end{equation}
Therefore
$$ I_{12} \leq 8 \epsilon^{-1}  n\mathrm{E} |Z^{>}_{1,n}| \sum_{j=q_{n}+1}^{\infty} \mathrm{E}| C_{j}|,$$
yielding $\lim_{n \to \infty}I_{12}=0$, since by Karamata's theorem  $ \lim_{n \to \infty} n\mathrm{E}|Z^{>}_{1,n}| = \alpha/(\alpha - 1)$ and we assumed (\ref{eq:convcond}) holds with $\delta=1$ in this case.
Thus
\begin{equation}\label{e:I1}
\lim_{n \to \infty}I_{1}=0.
\end{equation}

Now we consider $I_{2}$. Note that
\begin{eqnarray*}
  I_{2} & \leq & \Pr \bigg( \sup_{0 \leq t \leq 1} \bigg| \sum_{i=1}^{\lfloor nt \rfloor} \sum_{j=q_{n}+1}^{\infty} C_{j} \widetilde{Z}^{\leq}_{i-j,n} \bigg| > \frac{\epsilon}{4} \bigg) + \Pr \bigg( \sup_{0 \leq t \leq 1} \bigg| \sum_{i=1}^{\lfloor nt \rfloor} \sum_{j=q_{n}+1}^{\infty} C_{j} \widetilde{Z}^{>}_{i-j,n} \bigg| > \frac{\epsilon}{4} \bigg)\\[0.4em]
  &=:& I_{21}+I_{22}.
\end{eqnarray*}
Let
$$ W_{n}(t) =  \sum_{i=1}^{\lfloor nt \rfloor} \sum_{j=q_{n}+1}^{\infty} C_{j} \widetilde{Z}^{\leq}_{i-j,n}, \qquad t \in [0,1].$$
Take $0 \leq t_{1} < t_{2} \leq 1$, and consider (for $\rho >0$)
\begin{eqnarray*}\label{e:AT1}
 \nonumber \Pr ( |W_{n}(t_{2}) - W_{n}(t_{1})| > \rho) &=& \Pr \bigg( \bigg| \sum_{i=\lfloor nt_{1} \rfloor +1}^{\lfloor n t_{2} \rfloor} \sum_{j=q_{n}+1}^{\infty} C_{j} \widetilde{Z}^{\leq}_{i-j,n} \bigg| > \rho \bigg) \\[0.5em]
   &=&  \Pr \bigg( \bigg| \sum_{i=-\infty}^{\lfloor n t_{2} \rfloor -1} \widetilde{D}^{n, t_{1},t_{2}}_{i-q_{n}} \widetilde{Z}^{\leq}_{i-q_{n},n} \bigg| > \rho \bigg),
\end{eqnarray*}
where
$$  \widetilde{D}^{n, t_{1},t_{2}}_{i-q_{n}} =  \left\{ \begin{array}{cc}
                                   \displaystyle \sum_{j=q_{n}-i+ \lfloor nt_{1} \rfloor +1}^{q_{n}-i + \lfloor nt_{2} \rfloor}C_{j}, & \quad i \leq \lfloor nt_{1} \rfloor -1,\\[1.5em]
                                   \displaystyle \sum_{j=q_{n} +1}^{q_{n}-i + \lfloor nt_{2} \rfloor}C_{j}, & \quad i=\lfloor nt_{1} \rfloor, \ldots, \lfloor nt_{2} \rfloor -1,
                                 \end{array}\right.$$
and the last equality in (\ref{e:AT1}) follows by standard changes of variables and order of summation.
The sequence $(\widetilde{D}^{n, t_{1},t_{2}}_{i-q_{n}}\widetilde{Z}^{\leq}_{i-q_{n},n})_{i}$ is a martingale difference sequence, and hence the Bahr-Esseen inequality (which holds also for infinite sums, by the Fatou lemma) and Markov's inequality imply
\begin{eqnarray*}
  \Pr \bigg( \bigg| \sum_{i=-\infty}^{\lfloor n t_{2} \rfloor -1} \widetilde{D}^{n, t_{1},t_{2}}_{i-q_{n}} \widetilde{Z}^{\leq}_{i-q_{n},n} \bigg| > \rho \bigg) & \leq &  \rho^{-\eta} \Big( \frac{\eta}{\eta -1} \Big)^{\eta}  \sum_{i=-\infty}^{\lfloor n t_{2} \rfloor -1} \mathrm{E}|\widetilde{D}^{n,t_{1},t_{2}}_{i-q_{n}} \widetilde{Z}^{\leq}_{i-q_{n},n}|^{\eta} \\
   &=& \rho^{-\eta} \Big( \frac{\eta}{\eta -1} \Big)^{\eta} \mathrm{E} | \widetilde{Z}^{\leq}_{1,n}|^{\eta} \sum_{i=-\infty}^{\lfloor n t_{2} \rfloor -1} \mathrm{E}| \widetilde{D}^{n,t_{1},t_{2}}_{i-q_{n}}|^{\eta}.
\end{eqnarray*}
With the same argument as in (\ref{e:AT-tc}) we obtain
\begin{eqnarray}\label{e:AT-tc2}
  \nonumber \sum_{i= \lfloor nt_{1} \rfloor}^{\lfloor nt_{2} \rfloor -1}\mathrm{E}| \widetilde{D}^{n,t_{1},t_{2}}_{i-q_{n}}|^{\eta} & & \\[0.5em]
 \nonumber & \hspace*{-12em} \leq & \hspace*{-6em} \sum_{i= \lfloor nt_{1} \rfloor}^{\lfloor nt_{2} \rfloor -1} \mathrm{E} \bigg[ \bigg( \sum_{s=q_{n}+1}^{\infty} |C_{s}| \bigg)^{\eta-\delta} \sum_{j=q_{n}+1}^{\infty}|C_{j}|^{\delta} \bigg] \\[0.5em]
   & \hspace{-12em} = & \hspace*{-6em} ( \lfloor nt_{2} \rfloor - \lfloor nt_{1} \rfloor ) \mathrm{E} \bigg[ \bigg( \sum_{s=q_{n}+1}^{\infty} |C_{s}| \bigg)^{\eta-\delta} \sum_{j=q_{n}+1}^{\infty}|C_{j}|^{\delta} \bigg],
\end{eqnarray}
and
\begin{eqnarray}\label{e:AT-tc3}
  \nonumber \sum_{i= -\infty}^{\lfloor nt_{1} \rfloor -1} \mathrm{E}| \widetilde{D}^{n,t_{1}¸,t_{2}}_{i-q_{n}}|^{\eta} & &\\[0.5em]
  \nonumber  & \hspace*{-12em} \leq & \hspace*{-6em} \mathrm{E} \bigg[ \bigg( \sum_{s=q_{n}+1}^{\infty} |C_{s}| \bigg)^{\eta-\delta} \sum_{i= -\infty}^{\lfloor nt_{1} \rfloor -1} \sum_{j=q_{n}-i +\lfloor nt_{1} \rfloor +1}^{q_{n}-i+ \lfloor nt_{2} \rfloor}|C_{j}|^{\delta} \bigg] \\[0.5em]
   & \hspace*{-12em} \leq & \hspace{-6em} ( \lfloor nt_{2} \rfloor - \lfloor nt_{1} \rfloor ) \mathrm{E} \bigg[ \bigg( \sum_{s=q_{n}+1}^{\infty} |C_{s}| \bigg)^{\eta-\delta} \sum_{j=q_{n}+1}^{\infty}|C_{j}|^{\delta} \bigg],
\end{eqnarray}
where the last inequality follows from the fact that every $|C_{j}|^{\delta}$, for $j \geq q_{n}+1$, appears in the sum $\sum_{i= -\infty}^{\lfloor nt_{1} \rfloor -1} \sum_{j=q_{n}-i +\lfloor nt_{1} \rfloor +1}^{q_{n}-i+ \lfloor nt_{2} \rfloor}|C_{j}|^{\delta}$ at most $\lfloor nt_{2} \rfloor - \lfloor nt_{1} \rfloor$ times. Therefore
\begin{eqnarray*}
   \Pr \bigg( \bigg| \sum_{i=-\infty}^{\lfloor n t_{2} \rfloor -1} \widetilde{D}^{n,t_{1},t_{2}}_{i-q_{n}} \widetilde{Z}^{\leq}_{i-q_{n},n} \bigg| > \rho \bigg) &&  \\
   & \hspace*{-26em} \leq & \hspace*{-13em} 2 \rho^{-\eta} \Big( \frac{\eta}{\eta -1} \Big)^{\eta} n \mathrm{E} | \widetilde{Z}^{\leq}_{1,n}|^{\eta}  \frac{\lfloor nt_{2} \rfloor - \lfloor nt_{1} \rfloor}{n} \mathrm{E} \bigg[ \bigg( \sum_{s=q_{n}+1}^{\infty} |C_{s}| \bigg)^{\eta-\delta} \sum_{j=q_{n}+1}^{\infty}|C_{j}|^{\delta} \bigg].
\end{eqnarray*}
Since by (\ref{e:jensen}) and Karamata's theorem $\sup_{n} \{ n \mathrm{E} | \widetilde{Z}^{\leq}_{1,n}|^{\eta} \} < \infty$, and $(\lfloor nt_{2} \rfloor - \lfloor nt_{1} \rfloor)/n \leq 2(t_{2}-t_{1})$ for large $n$, it follows that
$$  \Pr \bigg( \bigg| \sum_{i=-\infty}^{\lfloor n t_{2} \rfloor -1} \widetilde{D}^{n,t_{1},t_{2}}_{i-q_{n}} \widetilde{Z}^{\leq}_{i-q_{n},n} \bigg| > \rho \bigg) \leq M \rho^{-\eta} (t_{2}-t_{1}) \mathrm{E} \bigg[ \bigg( \sum_{s=q_{n}+1}^{\infty} |C_{s}| \bigg)^{\eta-\delta} \sum_{j=q_{n}+1}^{\infty}|C_{j}|^{\delta} \bigg],$$
for some constant $M$ independent of $n$. Now by Theorem 2 in Avram and Taqqu~\cite{AvTa89} and the arguments in the proof of Proposition 4 in Avram and Taqqu~\cite{AvTa92} we conclude that
\begin{eqnarray*}
  \Pr \bigg( \sup_{0 \leq t \leq 1}\bigg| \sum_{i=-\infty}^{ \lfloor nt \rfloor -1} \widetilde{D}^{n,0,t}_{i-q_{n}} \widetilde{Z}^{\leq}_{i-q_{n},n} \bigg| > \rho \bigg) &&  \\[0.5em]
   & \hspace*{-22em} \leq & \hspace*{-11em} M' \rho^{-\eta} (\ln n)^{1+\eta}\mathrm{E} \bigg[ \bigg( \sum_{s=q_{n}+1}^{\infty} |C_{s}| \bigg)^{\eta-\delta} \sum_{j=q_{n}+1}^{\infty}|C_{j}|^{\delta} \bigg]
\end{eqnarray*}
for some constant $M'$ independent of $n$. From this and condition (\ref{e:vod1}), since $\ln n = O(\ln q_{n})$, it follows that
\begin{equation}\label{e:AT-tc5}
 \lim_{n \to \infty} I_{21} = \lim_{n \to \infty} \Pr \bigg( \sup_{0 \leq t \leq 1}\bigg| \sum_{i=-\infty}^{ \lfloor nt \rfloor -1} \widetilde{D}^{n,0,t}_{i-q_{n}} \widetilde{Z}^{\leq}_{i-q_{n},n} \bigg| > \frac{\epsilon}{4} \bigg)=0.
\end{equation}
Further, note that
\begin{eqnarray*}
   \Pr \bigg( \sup_{0 \leq t \leq 1} \bigg| \sum_{i=1}^{\lfloor nt \rfloor} \sum_{j=q_{n}+1}^{\infty} C_{j} \widetilde{Z}^{>}_{i-j,n} \bigg| > \frac{\epsilon}{4} \bigg) & \leq & \Pr \bigg( \sum_{i=1}^{n} \sum_{j=q_{n}+1}^{\infty} |C_{j} \widetilde{Z}^{>}_{i-j,n}| > \frac{\epsilon}{4} \bigg)  \\[0.5em]
  & \leq &  \Big( \frac{\epsilon}{4} \Big)^{-1} \sum_{i=1}^{n} \sum_{j=q_{n}+1}^{\infty} \mathrm{E}|C_{j} \widetilde{Z}^{>}_{i-j,n}|\\[0.5em]
  & = &  \Big( \frac{\epsilon}{4} \Big)^{-1} \mathrm{E}|\widetilde{Z}^{>}_{1,n}| \sum_{i=1}^{n} \sum_{j=q_{n}+1}^{\infty} \mathrm{E}|C_{j}|\\[0.5em]
   & = &  \Big( \frac{\epsilon}{4} \Big)^{-1} n \mathrm{E}|\widetilde{Z}^{>}_{1,n}| \sum_{j=q_{n}+1}^{\infty} \mathrm{E}|C_{j}|.
\end{eqnarray*}
By (\ref{e:AT-tc4}) and Karamata's theorem $\sup_{n} \{ n \mathrm{E} | \widetilde{Z}^{>}_{1,n}|^{\eta} \} < \infty$, and hence condition (\ref{eq:convcond}) (with $\delta=1$) implies
\begin{equation}\label{e:AT-tc6}
\lim_{n \to \infty} I_{22} =0.
\end{equation}
Now from (\ref{e:AT-tc5}) and (\ref{e:AT-tc6}) we get
\begin{equation}\label{e:I2}
 \lim_{n \to \infty} I_{2} =0.
\end{equation}
Therefore from (\ref{e:aproksqn}), (\ref{e:I1}) and (\ref{e:I2}) we conclude that
$$ \lim_{n \to \infty} \Pr \bigg( \sup_{0 \leq t \leq 1}|V_{n, q_{n}}(t) - V_{n}(t)|> \epsilon \bigg)=0.$$

Thus, in order to have $V_{n}(\,\cdot\,) \dto \widetilde{B}(\,\cdot\,) V(\,\cdot\,)$ in $D[0,1]$ with the $M_{2}$ topology, according to Slutsky's theorem (see Resnick~\cite{Resnick07}, Theorem 3.4), it remains to show
$V_{n, q_{n}}(\,\cdot\,) \dto \widetilde{B}(\,\cdot\,) V(\,\cdot\,)$ in $(D[0,1], d_{M_{2}})$ as $n \to \infty$. Note that we cannot simply use Theorem~\ref{t:FinMA} as we did in the case $\alpha \in (0,1)$, since now $q_{n}$ depends on $n$. By careful analysis of the proof of Theorem~\ref{t:FinMA} we see that relations that have to be checked, in order that the statement of Theorem~\ref{t:FinMA} remains valid if we replace $q$ by $q_{n}$, are (\ref{eq:setBn1}), (\ref{eq:setBn3}) and (\ref{eq:Tnend}) (with $C_{*} = \sum_{s=0}^{\infty}|C_{s}|$). Hence we have to establish the following relations
\begin{eqnarray*}
  & & \lim_{n \to \infty} (q_{n}-1)(2q_{n}-1) \Pr \bigg( \frac{|Z_{0}|}{a_{n}} > \frac{\epsilon}{3(2q_{n}-1)M} \bigg) =0   \\[0.8em]
   && \lim_{n \to \infty} \bigg[ q_{n} \Pr \bigg( \frac{|Z_{0}|}{a_{n}} > \frac{\epsilon}{2q_{n}M} \bigg) + nq_{n}^{2} \bigg( \Pr
            \bigg( \frac{|Z_{0}|}{a_{n}} > \frac{\epsilon}{2q_{n}M} \bigg) \bigg)^{2} \bigg]= 0 \\[0.8em]
 &&  \lim_{n \to \infty} n {q_{n} \choose 2}
 \bigg[ \Pr\left(
 \frac{|Z_{0}|}{a_{n}} >   \frac{\epsilon}{2 q_{n} M}
 \right)\bigg]^2 =0 ,
\end{eqnarray*}
for arbitrary $\epsilon >0$ and $M>0$.
For all of this, taking into consideration relation (\ref{e:regvar}), i.e. the regular variation property of $Z_{0}$, it suffices to show
$$ \lim_{n \to \infty} n q_{n}^{2} \bigg[\Pr \bigg( |Z_{0}| > \frac{a_{n}}{q_{n}} \bigg)\bigg]^{2}=0,$$
which holds by Lemma~\ref{l:tech} in Appendix.
Therefore we conclude $V_{n}(\,\cdot\,) \dto \widetilde{B}(\,\cdot\,) V(\,\cdot\,)$ in $(D[0,1], d_{M_{2}})$.
\\[-0.8em]

Case $\alpha =1$.
Since $Z_{i}$ is symmetric, note that $\widetilde{Z}^{\leq}_{i,n} = Z^{\leq}_{i,n}$ and $\widetilde{Z}^{>}_{i,n} = Z^{>}_{i,n}$. We proceed as in the case $\alpha \in (1,2)$ (with the notation from that case) to obtain
$\lim_{n \to \infty}I_{11}=0$ and $\lim_{n \to \infty}I_{21}=0$ . For $I_{12}$, by Markov's inequality and the triangle inequality $|\sum_{i=1}^{n}a_{i}|^{\gamma} \leq \sum_{i=1}^{n}|a_{i}|^{\delta}$ we obtain
$$ I_{12} \leq \Pr \bigg( \sum_{i=1}^{n} |C''_{q_{n}} \widetilde{Z}^{>}_{i-q_{n},n}| > \frac{\epsilon}{4} \bigg) \leq
\Big( \frac{\epsilon}{4} \Big)^{-\delta} n \mathrm{E}|Z^{>}_{1,n}|^{\delta} \sum_{i=q_{n}+1}^{\infty} \mathrm{E} |C_{j}|^{\delta}.$$
By Karamata's theorem  $ \lim_{n \to \infty} n\mathrm{E}|Z^{>}_{1,n}|^{\delta} = (1 - \delta)^{-1}$ and hence from (\ref{eq:convcond}) we have $\lim_{n \to \infty}I_{12}=0$. Similarly we obtain $\lim_{n \to \infty}I_{22}=0$. This all allows us to conclude $\lim_{n \to \infty}I_{1}=0$ and $\lim_{n \to \infty}I_{2}=0$, i.e.
$$ \lim_{n \to \infty} \Pr \bigg( \sup_{0 \leq t \leq 1}|V_{n, q_{n}}(t) - V_{n}(t)|> \epsilon \bigg)=0.$$ As before, Lemma~\ref{l:tech} from Appendix and the modified proof of Theorem~\ref{t:FinMA} (with $q$ replaced by $q_{n}$) imply $V_{n, q_{n}}(\,\cdot\,) \dto \widetilde{B}(\,\cdot\,) V(\,\cdot\,)$ in $(D[0,1], d_{M_{2}})$. Now the statement of the theorem follows by an application of Slutsky's theorem.

\end{proof}

\begin{rem}
When the sequence of coefficients $(C_{j})$ is deterministic, condition (\ref{e:vod1}) is not needed. This is known from the article of Basrak and Krizmani\'{c}~\cite{BaKr}, but their proof contain an error (i.e.~they used Lemma 2 from Avram and Taqqu~\cite{AvTa92}, but the conditions needed to use this lemma were not fulfilled). Therefore in the proposition below we improve the proof of Theorem 3.1 in Basrak and Krizmani\'{c}~\cite{BaKr} in the case $\alpha \in [1,2)$, thus showing that condition (\ref{e:vod1}) can be dropped if all coefficients of the moving average process are deterministic.

 For a deterministic sequence $(C_{j})$ condition (\ref{e:mod1}) can also be dropped since it is implied by (\ref{eq:convcond}). The latter in general does not hold when the coefficients $C_{j}$ are random. It can easily be seen by the following example. Take $\epsilon>0$ such that $\delta + \epsilon < \gamma$. Let $S= \sum_{j=1}^{\infty}j^{-(1+\delta+\epsilon)} < \infty$ and $S_{k}= S^{-1} \sum_{j=1}^{k}j^{-(1+\delta+\epsilon)}$, $k \in \mathbb{N}$ (with $S_{0}=0$).
Taking $\Pr$ to be the Lebesgue measure on the Borel subsets of $(0,1)$ and
$$C_{i}(\omega)=i\,1_{(S_{i-1},S_{i}]}(\omega), \qquad \omega \in (0,1), \ i \in \mathbb{N},$$
we obtain
$$ \sum_{i=1}^{\infty} \mathrm{E}|C_{i}|^{\delta}=  S^{-1} \sum_{i=1}^{\infty} i^{\delta} (S_{i}-S_{i-1}) = S^{-1} \sum_{i=1}^{\infty} \frac{1}{i^{1+\epsilon}} < \infty,$$
and
$$ \sum_{i=1}^{\infty} \mathrm{E}|C_{i}|^{\gamma}=  S^{-1} \sum_{i=1}^{\infty} i^{\gamma} (S_{i}-S_{i-1}) = S^{-1} \sum_{i=1}^{\infty} \frac{1}{i^{1+\delta +\epsilon - \gamma}} = \infty,$$
since $1+\delta+\epsilon-\gamma < 1$.
\end{rem}

\begin{prop}
Let $(Z_{i})_{i \in \mathbb{Z}}$ be an i.i.d.~sequence of regularly varying random variables with index $\alpha \in [1,2)$. Suppose that conditions $(\ref{e:oceknula})$ and $(\ref{e:sim})$ hold. Let $\{C_{i}, i=0,1,2,\ldots\}$ be a sequence of real numbers satisfying
\be\label{eq:convcond-prop}
 \sum_{j=0}^{\infty}|C_{j}|^{\delta} < \infty \qquad \textrm{for some} \ \delta < \alpha, \ 0 < \delta \leq 1,
\ee
 and
\be\label{eq:InfiniteMAcond-prop}
0 \le \sum_{i=0}^{s}C_{i} \Bigg/ \sum_{i=0}^{\infty}C_{i} \le 1 \qquad \textrm{for every} \ s=0, 1, 2 \ldots.
\ee
  Then
$$ V_{n}(\,\cdot\,) \dto C V(\,\cdot\,), \qquad n \to \infty,$$
in $D[0,1]$ endowed with the $M_{2}$ topology, where $C=\sum_{j=0}^{\infty}C_{j}$, $V$ is an $\alpha$--stable L\'{e}vy process with characteristic triple $(0, \mu, b)$, with $\mu$ as in $(\ref{eq:mu})$ and
$$ b = \left\{ \begin{array}{cc}
                                   0, & \quad \alpha = 1,\\[0.4em]
                                   (p-r)\frac{\alpha}{1-\alpha}, & \quad \alpha \in (1,2).
                                 \end{array}\right.$$
\end{prop}
\begin{proof}
Fix $q \in \mathbb{N}$, and define
$$ X_{i}^{q} = \sum_{j=0}^{q-1}C_{j}Z_{i-j} + C'_{q} Z_{i-q}, \qquad i \in \mathbb{Z},$$
where $C'_{q}= \sum_{j=q}^{\infty}C_{j}$,
and
$$ V_{n, q}(t) = \sum_{i=1}^{\floor{nt}} \frac{X_{i}^{q}}{a_{n}}, \qquad t \in [0,1].$$
  Since the coefficients $C_{0}, \ldots, C_{q-1}, C'_{q}$ satisfy condition $(\ref{eq:FiniteMAcond})$, an application of Theorem~\ref{t:FinMA}, adjusted to deterministic coefficients $C_{j}$, yields $V_{n, q}(\,\cdot\,) \dto C V(\,\cdot\,)$ in $(D[0,1], d_{M_{2}})$ as $n \to \infty$ (see also Theorem 2.1 in Basrak and Krizmani\'{c}~\cite{BaKr}). Therefore, in order to have $V_{n}(\,\cdot\,) \dto C V(\,\cdot\,)$ in $D[0,1]$ with the $M_{2}$ topology, by a generalization of Slutsky's theorem we have to show that for every $\epsilon >0$
$$ \lim_{q \to \infty} \limsup_{n \to \infty}\Pr[d_{M_{2}}(V_{n, q}, V_{n})> \epsilon]=0.$$
As before it suffices to show that
\begin{equation}\label{e:prop0}
 \lim_{q \to \infty} \limsup_{n \to \infty}\Pr \bigg( \sup_{0 \leq t \leq 1}|V_{n, q}(t) - V_{n}(t)|> \epsilon \bigg)=0.
\end{equation}
As in (\ref{e:aproksqn}) we have
\begin{eqnarray}\label{e:aproksqn-prop}
  \nonumber \Pr \bigg( \sup_{0 \leq t \leq 1}|V_{n, q}(t) - V_{n}(t)|> \epsilon \bigg) &&  \\[0.5em]
   & \hspace*{-36em} \leq & \hspace*{-18em}  \Pr \bigg( \sup_{1 \leq k \leq n} \bigg| \sum_{i=1}^{k} \frac{C''_{q} Z_{i-q}}{a_{n}} \bigg| > \frac{\epsilon}{2} \bigg) + \Pr \bigg( \sup_{1 \leq k \leq n} \bigg| \sum_{i=1}^{k} \sum_{j=q+1}^{\infty} \frac{C_{j} Z_{i-j}}{a_{n}} \bigg| > \frac{\epsilon}{2} \bigg),
\end{eqnarray}
where $C''_{q}=\sum_{i=q+1}^{\infty}C_{i}$. Take $\tau >0$ such that
$$  \left\{ \begin{array}{cc}
                                   \alpha - \tau > \delta, & \quad \textrm{if} \ \alpha = 1,\\[0.4em]
                                   \alpha - \tau > 1, & \quad \textrm{if} \ \alpha \in (1,2).
                                 \end{array}\right.$$
  By (\ref{eq:convcond-prop}) we see that for large $q$ it holds that $|C_{j}|^{\delta} < 1$ for all $j \geq q+1$, which implies $|C_{j}|^{\alpha-\tau} \leq |C_{j}|^{\delta}$. Hence $\sum_{j=0}^{\infty}|C_{j}|^{\alpha-\tau} < \infty$. Similarly $\sum_{j=0}^{\infty}|C_{j}| < \infty$. This implies that for large $q$ we have $|C''_{q}| < 1$, which allows us to apply Lemma 2 in Avram and Taqqu~\cite{AvTa92} to the first term on the right-hand side of (\ref{e:aproksqn-prop}), to obtain (for large $q$)
  \begin{eqnarray}\label{e:prop1}
   \nonumber \Pr \bigg( \sup_{1 \leq k \leq n} \bigg| \sum_{i=1}^{k} \frac{C''_{q} Z_{i-q}}{a_{n}} \bigg| > \frac{\epsilon}{2} \bigg) & \leq & M \Big( \frac{\epsilon}{2} \Big)^{-(\alpha+\tau)} \frac{1}{n} \sum_{i=1}^{n}|C''_{q}|^{\alpha-\tau} \\[0.5em]
     &=& M \Big( \frac{\epsilon}{2} \Big)^{-(\alpha+\tau)} |C''_{q}|^{\alpha-\tau}
  \end{eqnarray}
where $M$ is a constant independent of $n$ and $q$. Using the following inequalities
\begin{equation}\label{e:ineq}
 \left\{ \begin{array}{cl}
 | \sum_{i=1}^{m}a_{i}|^{s} \leq \sum_{i=1}^{m}|a_{i}|^{s}, &  \qquad \textrm{if} \ s \leq 1,\\[0.6em]
 |\sum_{i=1}^{m}a_{i}|^{s} \leq \sum_{i=1}^{m}|a_{i}|, & \qquad \textrm{if} \ s>1 \ \textrm{and} \ | \sum_{i=1}^{m} a_{i}| < 1,
\end{array}\right.
\end{equation}
we have
$$  |C''_{q}|^{\alpha-\tau} \leq
 \left\{ \begin{array}{lc}
  \sum_{j=q+1}^{\infty}|C_{j}|^{\alpha-\tau} , &  \qquad \textrm{if} \ \alpha =1,\\[0.7em]
  \sum_{j=q+1}^{\infty}|C_{j}|, & \qquad \textrm{if} \ \alpha \in (1,2),
\end{array}\right.$$
yielding $\lim_{q \to \infty}|C''_{q}|^{\alpha-\tau} =0$.
Now from (\ref{e:prop1}) we obtain
\begin{equation}\label{e:prop1-0}
 \lim_{q \to \infty} \limsup_{n \to \infty} \Pr \bigg( \sup_{1 \leq k \leq n} \bigg| \sum_{i=1}^{k} \frac{C''_{q} Z_{i-q}}{a_{n}} \bigg| > \frac{\epsilon}{2} \bigg) =0.
\end{equation}
Note that the second term on the right-hand side of (\ref{e:aproksqn-prop}) is bounded above by
$$ \Pr \bigg( \sup_{1 \leq k \leq n} \bigg| \sum_{i=1}^{k} \sum_{j=q+1}^{\infty} \frac{C_{j}^{+} Z_{i-j}}{a_{n}} \bigg| > \frac{\epsilon}{4} \bigg) + \Pr \bigg( \sup_{1 \leq k \leq n} \bigg| \sum_{i=1}^{k} \sum_{j=q+1}^{\infty} \frac{C_{j}^{-} Z_{i-j}}{a_{n}} \bigg| > \frac{\epsilon}{4} \bigg),$$
where $C_{j}^{+} = C_{j} \vee 0$ and $C_{j}^{-}= (-C_{j}) \vee 0$. In the sequel we consider only the first of these two probabilities since the other one can be handled in the same manner.

Assume first $\alpha \in (1,2)$. Recall
$ \widetilde{Z}^{\leq}_{i,n}$ and
$\widetilde{Z}^{>}_{i,n}$ from the proof of Theorem~\ref{t:InfMA}, and note
\begin{eqnarray}\label{e:prop1-1}
  \nonumber \Pr \bigg( \sup_{1 \leq k \leq n} \bigg| \sum_{i=1}^{k} \sum_{j=q+1}^{\infty} \frac{C_{j}^{+} Z_{i-j}}{a_{n}} \bigg| > \frac{\epsilon}{4} \bigg) && \\[0.5em]
  \nonumber  & \hspace*{-32em} \leq & \hspace*{-16em} \Pr \bigg( \sup_{1 \leq k \leq n} \bigg| \sum_{i=1}^{k} \sum_{j=q+1}^{\infty} C_{j}^{+} \widetilde{Z}^{\leq}_{i-j,n} \bigg| > \frac{\epsilon}{8} \bigg) +  \Pr \bigg( \sup_{1 \leq k \leq n} \bigg| \sum_{i=1}^{k} \sum_{j=q+1}^{\infty} C_{j}^{+} \widetilde{Z}^{>}_{i-j,n} \bigg| > \frac{\epsilon}{8} \bigg)\\[-0.2em]
   & &
\end{eqnarray}
Since the coefficients $C_{j}^{+}$ are nonnegative, the moving average processes
$$Y_{i,n,q}^{\leq} := \sum_{j=q+1}^{\infty} C_{j}^{+} \widetilde{Z}^{\leq}_{i-j,n}, \qquad i=1,2,\ldots,$$
are associated, as nondecreasing functions of independent random variables (see Esary et al.~\cite{EsPrWa67}). Thus the sequence $( \sum_{i=1}^{k}Y_{i,n,q}^{\leq})_{k}$ is a demimartingale (see Section 2.1 in Prakasa Rao~\cite{Ra12}), and hence by Markov's inequality and the maximal inequality for demimartingales
$$ \mathrm{E} \bigg( \sup_{1 \leq k \leq n}|S_{k}| \bigg)^{\kappa} \leq \Big( \frac{\kappa}{\kappa-1} \Big)^{\kappa} \mathrm{E} |S_{n}|^{\kappa},$$
which holds for $\kappa >1$ and $(S_{k})_{k}$ a demimartingale
 (see for example Corollary 2.4 in Wang et al.~\cite{WHZY10})  we obtain
\begin{equation}\label{e:prop2}
 \Pr \bigg( \sup_{1 \leq k \leq n} \bigg| \sum_{i=1}^{k} \sum_{j=q+1}^{\infty} C_{j}^{+} \widetilde{Z}^{\leq}_{i-j,n} \bigg| > \frac{\epsilon}{8} \bigg) \leq \Big( \frac{\epsilon}{8} \Big)^{-(\alpha+\tau)} \Big( \frac{\alpha+\tau}{\alpha+\tau-1} \Big)^{\alpha+\tau} \mathrm{E} \bigg| \sum_{i=1}^{n}Y_{i,n,q}^{\leq} \bigg|^{\alpha+\tau}
\end{equation}
and similarly
\begin{equation}\label{e:prop3}
 \Pr \bigg( \sup_{1 \leq k \leq n} \bigg| \sum_{i=1}^{k} \sum_{j=q+1}^{\infty} C_{j}^{+} \widetilde{Z}^{>}_{i-j,n} \bigg| > \frac{\epsilon}{8} \bigg) \leq \Big( \frac{\epsilon}{8} \Big)^{-(\alpha-\tau)} \Big( \frac{\alpha-\tau}{\alpha-\tau-1} \Big)^{\alpha-\tau} \mathrm{E} \bigg| \sum_{i=1}^{n}Y_{i,n,q}^{>}\bigg|^{\alpha-\tau},
\end{equation}
where $Y_{i,n,q}^{>} := \sum_{j=q+1}^{\infty} C_{j}^{+} \widetilde{Z}^{>}_{i-j,n}$.
By standard changes of variables and order of summation we have
$$ \sum_{i=1}^{n}Y_{i,n,q}^{\leq} = \sum_{i=-\infty}^{n-1} \bigg( \sum_{j=q+1+(-i) \vee 0}^{q+n-i}C_{j}^{+} \bigg) \widetilde{Z}^{\leq}_{i-q,n}.$$
Note that $( ( \sum_{j=q+1+(-i) \vee 0}^{q+n-i}C_{j}^{+}) \widetilde{Z}^{\leq}_{i-q,n} )_{i}$ is a martingale difference sequence, and thus by the Bahr-Esseen inequality we obtain
\begin{eqnarray*}
  \mathrm{E} \bigg| \sum_{i=1}^{n}Y_{i,n,q}^{\leq} \bigg|^{\alpha+\tau} & \leq & 2 \sum_{i=-\infty}^{n-1} \bigg( \sum_{j=q+1+(-i) \vee 0}^{q+n-i}C_{j}^{+} \bigg)^{\alpha+\tau} \mathrm{E} |\widetilde{Z}^{\leq}_{i-q,n}|^{\alpha+\tau}.
\end{eqnarray*}
 Noting that for large $q$, $ \sum_{j=q+1+(-i) \vee 0}^{q+n-i}C_{j}^{+} < 1$, the second inequality in (\ref{e:ineq})
 yields that (for large $q$)
 $$\mathrm{E} \bigg| \sum_{i=1}^{n}Y_{i,n,q}^{\leq} \bigg|^{\alpha+\tau} \leq 2  |\widetilde{Z}^{\leq}_{1,n}|^{\alpha+\tau} \sum_{i=-\infty}^{n-1} \sum_{j=q+1+(-i) \vee 0}^{q+n-i}C_{j}^{+}.$$
 Now note that every $C_{j}^{+}$, for $j \geq q+1$, appears in the sum  $\sum_{i=-\infty}^{n-1} \sum_{j=q+1+(-i) \vee 0}^{q+n-i}C_{j}^{+}$ at most $n$ times, and therefore
  $$\mathrm{E} \bigg| \sum_{i=1}^{n}Y_{i,n,q}^{\leq} \bigg|^{\alpha+\tau} \leq 2 n  |\widetilde{Z}^{\leq}_{1,n}|^{\alpha+\tau} \sum_{j=q+1}^{\infty}C_{j}^{+}.$$
Similarly we obtain
 $$\mathrm{E} \bigg| \sum_{i=1}^{n}Y_{i,n,q}^{>} \bigg|^{\alpha-\tau} \leq 2 n  |\widetilde{Z}^{>}_{1,n}|^{\alpha-\tau} \sum_{j=q+1}^{\infty}C_{j}^{+}.$$
 Jensen's inequality, as in (\ref{e:jensen}), yields
 $$  \mathrm{E}|\widetilde{Z}^{\leq}_{1,n}|^{\alpha+\tau}  \leq 2^{\alpha+\tau +1} \mathrm{E}|Z^{\leq}_{1,n}|^{\alpha+\tau},$$
 and similarly
  $$  \mathrm{E}|\widetilde{Z}^{>}_{1,n}|^{\alpha-\tau}  \leq 2^{\alpha-\tau +1} \mathrm{E}|Z^{>}_{1,n}|^{\alpha-\tau}.$$
Collecting all these facts, from (\ref{e:prop2}) and (\ref{e:prop3}) we obtain, for large $q$,
\begin{eqnarray}\label{e:prop4}
 \nonumber \Pr \bigg( \sup_{1 \leq k \leq n} \bigg| \sum_{i=1}^{k} \sum_{j=q+1}^{\infty} C_{j}^{+} \widetilde{Z}^{\leq}_{i-j,n} \bigg| > \frac{\epsilon}{8} \bigg) & &\\[0.5em]
   & \hspace*{-28em} \leq & \hspace*{-14em} 2^{\alpha+\tau+2} \Big( \frac{\epsilon}{8} \Big)^{-(\alpha+\tau)} \Big( \frac{\alpha+\tau}{\alpha+\tau-1} \Big)^{\alpha+\tau} n  |Z^{\leq}_{1,n}|^{\alpha+\tau} \sum_{j=q+1}^{\infty}C_{j}^{+}
\end{eqnarray}
and
\begin{eqnarray}\label{e:prop5}
  \nonumber \Pr \bigg( \sup_{1 \leq k \leq n} \bigg| \sum_{i=1}^{k} \sum_{j=q+1}^{\infty} C_{j}^{+} \widetilde{Z}^{>}_{i-j,n} \bigg| > \frac{\epsilon}{8} \bigg) & &\\[0.5em]
   & \hspace*{-28em} \leq & \hspace*{-14em} 2^{\alpha-\tau+2} \Big( \frac{\epsilon}{8} \Big)^{-(\alpha-\tau)} \Big( \frac{\alpha-\tau}{\alpha-\tau-1} \Big)^{\alpha-\tau} n  |Z^{>}_{1,n}|^{\alpha-\tau} \sum_{j=q+1}^{\infty}C_{j}^{+}.
\end{eqnarray}
From (\ref{e:prop1-1}), (\ref{e:prop4}) and (\ref{e:prop5}) we see that for some positive constant $M'$ the following inequality holds for large $q$
$$  \Pr \bigg( \sup_{1 \leq k \leq n} \bigg| \sum_{i=1}^{k} \sum_{j=q+1}^{\infty} \frac{C_{j}^{+} Z_{i-j}}{a_{n}} \bigg| > \frac{\epsilon}{4} \bigg) \leq M' (n |Z^{\leq}_{1,n}|^{\alpha+\tau} + n  |Z^{>}_{1,n}|^{\alpha-\tau}) \sum_{j=q+1}^{\infty}C_{j}^{+}.$$
By Karamata's theorem $ n\mathrm{E}|Z^{\leq}_{1,n}|^{\alpha+\tau}  \to \alpha/\tau$
and
$ n\mathrm{E}|Z^{>}_{1,n}|^{\alpha-\tau} \to \alpha/\tau$, as $n \to \infty$. Therefore, since $\sum_{j=q+1}^{\infty}C_{j}^{+} \leq \sum_{j=q+1}^{\infty}|C_{j}| \to 0$ as $q \to \infty$, we have
$$ \lim_{q \to \infty} \limsup_{n \to \infty} \Pr \bigg( \sup_{1 \leq k \leq n} \bigg| \sum_{i=1}^{k} \sum_{j=q+1}^{\infty} \frac{C_{j}^{+} Z_{i-j}}{a_{n}} \bigg| > \frac{\epsilon}{4} \bigg)=0.$$
Hence we conclude
\begin{equation}\label{e:prop6}
\lim_{q \to \infty} \limsup_{n \to \infty} \Pr \bigg( \sup_{1 \leq k \leq n} \bigg| \sum_{i=1}^{k} \sum_{j=q+1}^{\infty} \frac{C_{j} Z_{i-j}}{a_{n}} \bigg| > \frac{\epsilon}{2} \bigg)=0.
\end{equation}
Now, from (\ref{e:aproksqn-prop}), (\ref{e:prop1-0}) and (\ref{e:prop6}) follows
(\ref{e:prop0}), which means that $V_{n}(\,\cdot\,) \dto C V(\,\cdot\,)$ in $D[0,1]$ with the $M_{2}$ topology.

Assume now $\alpha=1$.
Relation (\ref{e:prop2}) holds also in this case, but for (\ref{e:prop3}) we need a different argument since $\alpha-\tau < 1$, and thus we can not use the maximal inequality for demimartingales. By Markov's inequality and the first inequality in (\ref{e:ineq}) we have
\begin{eqnarray*}
 \Pr \bigg( \sup_{1 \leq k \leq n} \bigg| \sum_{i=1}^{k} \sum_{j=q+1}^{\infty} C_{j}^{+} \widetilde{Z}^{>}_{i-j,n} \bigg| > \frac{\epsilon}{8} \bigg) & \leq &\Pr \bigg( \sum_{i=1}^{n} \sum_{j=q+1}^{\infty} C_{j}^{+} |\widetilde{Z}^{>}_{i-j,n}|  > \frac{\epsilon}{8} \bigg)\\[0.5em]
 & \leq & \Big( \frac{\epsilon}{8} \Big)^{-(\alpha-\tau)}  \mathrm{E} \bigg( \sum_{i=1}^{n} \sum_{j=q+1}^{\infty} C_{j}^{+} |\widetilde{Z}^{>}_{i-j,n}| \bigg)^{\alpha-\tau},\\[0.5em]
 & \leq & \Big( \frac{\epsilon}{8} \Big)^{-(\alpha-\tau)} \mathrm{E} |\widetilde{Z}^{>}_{1,n}|^{\alpha-\tau} \sum_{i=1}^{n} \sum_{j=q+1}^{\infty} (C_{j}^{+})^{\alpha-\tau}\\[0.5em]
 & \leq & \Big( \frac{\epsilon}{8} \Big)^{-(\alpha-\tau)} n \mathrm{E} |\widetilde{Z}^{>}_{1,n}|^{\alpha-\tau} \sum_{j=q+1}^{\infty} |C_{j}|^{\alpha-\tau}.
\end{eqnarray*}
From the symmetry of $Z_{1}$, Karamata's theorem and (\ref{eq:niz}) we obtain, as $n \to \infty$,
$$ n\mathrm{E}|\widetilde{Z}^{>}_{1,n}|^{\alpha-\tau}  = n\mathrm{E}|Z^{>}_{1,n}|^{\alpha-\tau} = \frac{\mathrm{E}(|Z_{1}|^{\alpha-\tau}1_{\{ |Z_{1}| > a_{n}\}})}{a_{n}^{\alpha-\tau} \Pr(|Z_{1}| > a_{n})} \cdot n \Pr(|Z_{1}| > a_{n}) \to \frac{\alpha}{\tau}.$$
Therefore, since $\lim_{q \to \infty} \sum_{j=q+1}^{\infty} |C_{j}|^{\alpha-\tau} =0$, we have
$$ \lim_{q \to \infty} \limsup_{n \to \infty} \Pr \bigg( \sup_{1 \leq k \leq n} \bigg| \sum_{i=1}^{k} \sum_{j=q+1}^{\infty} C_{j}^{+} \widetilde{Z}^{>}_{i-j,n} \bigg| > \frac{\epsilon}{8} \bigg) =0,$$
and as in the case $\alpha \in (1,2)$ it follows that $V_{n}(\,\cdot\,) \dto C V(\,\cdot\,)$ in $D[0,1]$ with the $M_{2}$ topology. This completes the proof.
\end{proof}

\section{Appendix}

We provide a technical result used in the proof of Theorem~\ref{t:InfMA}.

\begin{lem}\label{l:tech}
 Let $Z_{1}$ be a regularly varying random variable with index $\alpha \in [1,2)$ and $(a_{n})$ a sequence of positive real numbers such that $(\ref{eq:niz})$ holds. Let $q_{n}= \lfloor n^{1/10} \rfloor,\,n \in \mathbb{N}$. Then
 $$ \lim_{n \to \infty} n q_{n}^{2} \bigg[\Pr \bigg( |Z_{1}| > \frac{a_{n}}{q_{n}} \bigg)\bigg]^{2}=0.$$
\end{lem}
\begin{proof}
By (\ref{e:regvar}) and (\ref{eq:niz}) we have
\begin{equation}\label{e:App1}
 \lim_{n \to \infty} n a_{n}^{-\alpha} L(a_{n}) =1.
\end{equation}
Since $L$ is a slowly varying function, it holds that for all $s>0$ and $t \in \mathbb{R}$, as $x \to \infty$, $x^{s}[L(x)]^{t} \to \infty$ and $x^{-s}[L(x)]^{t} \to 0$ (Bingham et al.~\cite{BiGoTe89}, Proposition 1.3.6). Hence $a_{n}^{2-\alpha}L(a_{n}) \to \infty$ as $n \to \infty$, and since by (\ref{e:App1})
$$ \lim_{n \to \infty} \frac{n}{a_{n}^{2}}\,a_{n}^{2-\alpha}L(a_{n})=1,$$
it follows that $n/a_{n}^{2} \to 0$ as $n \to \infty$. This yields
$$ \frac{a_{n}}{q_{n}} = \sqrt{\frac{a_{n}^{2}}{n}} \cdot \frac{\sqrt{n}}{q_{n}} \to \infty \qquad \textrm{as} \ n \to \infty,$$
since by the definition of the sequence $(q_{n})$, $\sqrt{n}/q_{n} \to \infty$. Thus for $u>0$, $M_{n}(u) := (a_{n}/q_{n})^{-u} [L(a_{n}/q_{n})]^{2} \to 0$ as $n \to \infty$.

From (\ref{e:regvar}) we obtain
$$ n q_{n}^{2} \bigg[\Pr \bigg( |Z_{1}| > \frac{a_{n}}{q_{n}} \bigg)\bigg]^{2} = n q_{n}^{2} \Big( \frac{a_{n}}{q_{n}} \Big)^{-2\alpha} \bigg[ L \Big( \frac{a_{n}}{q_{n}} \Big) \bigg]^{2} =
n q_{n}^{2} \Big( \frac{a_{n}}{q_{n}} \Big)^{-2\alpha +u} M_{n}(u).$$
By (\ref{e:App1}) we have
$$ a_{n}^{\alpha} \geq K nL(a_{n})$$
for some positive constant $K$ independent of $n$, and hence taking some $v>0$ such that $u+v < 2\alpha$ we obtain
\begin{eqnarray}\label{e:App2}
  \nonumber n q_{n}^{2} \bigg[\Pr \bigg( |Z_{1}| > \frac{a_{n}}{q_{n}} \bigg)\bigg]^{2} & = &
   n q_{n}^{2+2\alpha-u} (a_{n}^{\alpha})^{-2+(u+v)/\alpha} a_{n}^{-v} M_{n}(u)\\[0.4em]
   \nonumber & \leq & K^{-2+(u+v)\alpha} q_{n}^{2+2\alpha-u} n^{-1+(u+v)/\alpha} a_{n}^{-v} [L(a_{n})]^{-2+(u+v)/\alpha} M_{n}(u) \\[0.6em]
   \nonumber &=& q_{n}^{2+2\alpha-u} n^{-1+(u+v)/\alpha} M_{n}(u,v)\\[0.6em]
   & \leq & n^{(2+2\alpha -u)/10 -1+(u+v)/\alpha} M_{n}(u,v).
\end{eqnarray}
where
$$ M_{n}(u,v) := K^{-2+(u+v)\alpha}a_{n}^{-v} [L(a_{n})]^{-2+(u+v)/\alpha} M_{n}(u) \to 0, \qquad \textrm{as} \ n \to \infty.$$
Now let $u=1/5$ and $v=1/5$, and note that for this choice of $u$ and $v$ it holds that
$$ \frac{2+2\alpha-u}{10}-1+\frac{u+v}{\alpha} \leq -\frac{1}{50} < 0.$$
Therefore from (\ref{e:App2}) we obtain
 $$ \lim_{n \to \infty} n q_{n}^{2} \bigg[\Pr \bigg( |Z_{1}| > \frac{a_{n}}{q_{n}} \bigg)\bigg]^{2}=0.$$
\end{proof}

\section*{Acknowledgment}
 This work has been supported in part by Croatian Science Foundation under the project 3526 and by University of Rijeka research grants 13.14.1.2.02 and 17.15.2.2.01.

%\bibliographystyle{amsplain}
%%
% requires a BiBTeX file sample.bib
%\bibliography{sample}

\end{document}